\documentclass[smallextended,numbook,final]{svjour3}
\renewcommand{\normalsize}{\fontsize{10pt}{\baselineskip}\selectfont}
\newcommand{\wuhao}{\fontsize{9pt}{\baselineskip}\selectfont}
\usepackage{makecell}
\usepackage[titletoc,title]{appendix}
\usepackage[numbers]{natbib}
\usepackage{amsmath,url,cases,supertabular}
\usepackage{amssymb,mathtools,multirow}
\usepackage{enumerate,makecell}
\usepackage{amsfonts}
\usepackage{graphicx}
\usepackage{color,verbatim}
\usepackage{mathrsfs,subeqnarray,subfigure}
\usepackage{listings}
\usepackage[left=0.9in,top=0.98in,right=0.9in,bottom=0.98in]{geometry}
\usepackage{array}

\usepackage{xcolor}

\newcommand{\BracKern}{\kern-\nulldelimiterspace}
\makeatletter
    \newcommand{\@Brac}[3]{
 \mathopen{\color{red}\left#1\vphantom{#2}\BracKern\right.}
   #2
   \mathclose{\color{red}\left.\BracKern\vphantom{#2}\right#3}
  }
\newcommand{\bracr}[1]{\@Brac{(}{#1}{)}}%
\newcommand{\bracs}[1]{\@Brac{[}{#1}{]}}%
\newcommand{\bracc}[1]{\@Brac{\{}{#1}{\}}}%
\newcommand{\bracsred}[1]{\@Brac{[}{#1}{]}}%
\newcommand{\bracrred}[1]{\@Brac{(}{#1}{)}}%
\makeatother
\newcommand{\upcite}[1]{\textsuperscript{\textsuperscript{\cite{#1}}}}
\newcommand{\be}{\begin{equation}}
\newcommand{\ee}{\end{equation}}
\newcommand{\bea}{\begin{eqnarray}}
\newcommand{\eea}{\end{eqnarray}}

\newcommand{\stief}{\mathcal{S}_{n,p}}

\newcommand{\mF}{\mathcal{F}}
\newcommand{\mL}{\mathcal{L}}
\newcommand{\mbR}{\mathbb{R}}
\newcommand{\mbC}{\mathbb{C}}
\newcommand{\mT}{\mathcal{T}}

\newcommand{\mA}{\mathcal{A}}
\newcommand{\mN}{\mathcal{N}}
\newcommand{\mR}{\mathcal{R}}
\newcommand{\wF}{\widetilde{F}}
\newcommand{\wR}{\widetilde{R}}
\newcommand{\wN}{\widetilde{N}}

\newcommand{\mD}{\mathcal{D}}
\newcommand{\Drho}{D_{\rho}}
\newcommand{\argmin}{\mathop{\mathrm{arg\,min}}}
\newcommand{\argmax}{\mathop{\mathrm{arg\,max}}}

\newcommand{\Drhok}{D_{\rho, k}}
\newcommand{\mP}{\mathcal{P}}

\newcommand{\mtr}{\mbox{tr}}

\newcommand{\reff}[1]{(\ref{#1})}

\newcommand{\nn}{\nonumber}

\newcommand{\Diag}{\text{Diag}}
\newcommand{\proj}{\mathcal{P}_{\mathcal{T}}}
\spnewtheorem{algorithm}{Algorithm}[section]{\bf}{\rm}
\spnewtheorem{assumption}{Assumption}[section]{\bf}{\rm}
\newcommand{\ba}{\begin{array}}
\newcommand{\ea}{\end{array}}

\newcommand{\nF}{\nabla \mathcal{F}}
\newcommand{\Rmn}[1]{\uppercase\expandafter{\romannumeral#1}}
\renewcommand{\theenumi}{\roman{enumi}}

\newcommand{\revmini}[1]{{\color{black}{#1}}}
\newcommand{\rev}[1]{{\color{black}{#1}}}
\newcommand{\rrev}[1]{{\color{black}{#1}}}
\newcommand{\brev}[1]{{\color{black}{#1}}}
\newcommand{\lastrev}[1]{{\color{black}{#1}}}
\renewcommand{\smartqed}{\hfill{\qed}}
\newcommand{\msF}{\mathsf{F}}
\newcommand{\msT}{\mathsf{T}}
\begin{document}
\title{A Framework of Constraint Preserving Update Schemes for Optimization on Stiefel Manifold \thanks{This work was partly supported by the Chinese NSF
grants (nos. 10831106 and 81211130105), the CAS grant (no. kjcx-yw-s7-03) and the China
National Funds for Distinguished Young Scientists (no. 11125107).}
}
\author{Bo Jiang \and Yu-Hong Dai}
\institute{
  B. Jiang \and Y.-H. Dai
\at LSEC, ICMSEC, Academy of Mathematics and Systems Sciences, Chinese Academy of Sciences, Beijing 100190, CHINA.
\and
B. Jiang
\\\email{jiangbo@lsec.cc.ac.cn}
\vskip5.23pt
Y.-H. Dai
\\\email{dyh@lsec.cc.ac.cn}
\\ Corresponding Author.
}
\date{Revised August 25, 2014}
\maketitle
\thispagestyle{empty}
\begin{abstract}
 This paper considers optimization problems on the Stiefel manifold $X^{\msT}X=I_p$, where $X\in \mathbb{R}^{n \times p}$ is the
 variable and $I_p$ is the $p$-by-$p$ identity matrix. A framework of constraint preserving update schemes is proposed by decomposing each feasible point into the range space  of $X$ and the null space of $X^{\msT}$. While this general framework can unify many existing schemes, a new update scheme with low complexity cost is also discovered.  Then we study a feasible Barzilai-Borwein-like method under the new update scheme. The global convergence of  the method is established with an adaptive nonmonotone line search. The numerical tests on the  nearest low-rank correlation matrix problem, the Kohn-Sham total energy minimization  and a specific problem from statistics demonstrate the efficiency of the new method. In particular, the new method performs \rev{remarkably well for the nearest low-rank correlation matrix problem in terms of  \rrev{speed} and solution quality} and is considerably competitive with the widely used SCF iteration for the Kohn-Sham total energy minimization.

\end{abstract}



 \section{Introduction}
  In this paper, we consider general feasible methods for optimization on the Stiefel manifold,
\be \label{prob:Stiefel}
  \min_{X \in \mathbb{R}^{n \times p}} \, \mathcal{F}(X),\ \ \mbox{s.t.} \ \   X^{\msT}X=I_p,
\ee
where $I_p$ is the $p$-by-$p$ identity matrix and $\mathcal{F}(X)\colon\mbR^{n \times p} \rightarrow \mbR$ is a differentiable function.
The feasible set $\stief \coloneqq\{X\in \mathbb{R}^{n \times p} : X^{\msT}X=I_p\}$ is referred to as the Stiefel manifold, which was due to
Stiefel in 1935 \cite{stiefel1935richtungsfelder}.

Problem \reff{prob:Stiefel} captures many applications, for instance, nearest low-rank correlation matrix problem   \cite{grubisic2007efficient, pietersz2004rank, Rebonato1999most}, linear eigenvalue problem \cite{golub1996matrix, saad1992numerical},  Kohn-Sham \brev{total} energy minimization \cite{yang2009kssolv},  orthogonal Procrustes problem \cite{elden1999procrustes, schonemann1966generalized},
maximization of sums of heterogeneous quadratic functions from statistics \cite{bolla1998extrema, rapcsak2001minimization}, sparse  principal component analysis \cite{aspremont2007direct, journee2010generalized, zou2006sparse}, leakage interference minimization \cite{liu2011complexity, peters2009interference}, joint diagonalization(blind source separation) \cite{joho2002joint, theis2009soft} \rev{and determining a minimal set of localized orbitals from atomic chemistry \cite{borsdorf2012algorithm}}. For other applications, we refer interested readers to \cite{edelman1998geometry, wen2013feasible} and the references therein. In general, problem \reff{prob:Stiefel} is difficult to solve due to the nonconvexity of the orthogonality constraint. In fact, some of the above examples, including the maxcut problem  and the leakage interference minimization \cite{liu2011complexity}, are NP-hard.

With the wide applicability  and fundamental difficulty, problem \reff{prob:Stiefel} has attracted many researchers. Based on the geometric structure, Rapcs\'ak \cite{rapcsak2001minimization, rapcsak2002minimization} reformulated it   as a smooth nonlinear program by introducing a new  coordinate representation.  From the point of view of manifold, some authors proposed a variety of feasible algorithms to  solve problem \reff{prob:Stiefel}. These algorithms include steepest descent methods \cite{abrudan2008steepest, absil2008optimization, manton2002optimization, nishimori2005learning}, Barzilai-Borwein (BB) method \cite{wen2013feasible}, conjugate gradient methods \cite{abrudan2009conjugate, absil2008optimization,  edelman1998geometry}, trust region methods \cite{absil2008optimization, yang2007trust}, Newton methods \cite{absil2008optimization, edelman1998geometry}, quasi-Newton methods \cite{savas2010quasi} and subspace methods \cite{yang2006constrained}. Unlike the unconstrained case, it is not trivial to keep the whole iterations in the Stiefel manifold and the concept of retraction has played an important role (see \cite[Theorem 15]{absil2012projection} and \cite[Definition 4.1.1]{absil2008optimization} for a detailed description on \brev{retractions}). \rev{Simply speaking, the retraction defines an update scheme which preserves the orthogonality constraint.} The existing update schemes \rev{employed by the aforementioned methods are some specific choices of retractions. They can be classified} into two types: geodesic-like and projection-like update schemes.  Briefly speaking, geodesic-like update schemes preserve the constraint by moving a point along the geodesic or quasi-geodesic while projection-like update schemes do so by (approximately) projecting a point into the constraint. We will delve into the details of these update schemes in \S\ref{section:reviewupdateschems}.

In this paper, we \rev{firstly} develop a framework of constraint preserving update schemes based on a novel idea \lastrev{of} decomposing each feasible point into the range space of $X$ and the null space of $X^{\msT}$. This framework \lastrev{cannot} only unify most existing schemes including \rrev{a kind of geodesic}, gradient projection, Manton's projection, polar decomposition, QR factorization and Wen-Yin  update  schemes (they will be mentioned in \S\ref{section:reviewupdateschems}), but also leads to the discovery of a new scheme with low complexity cost. 

\rev{Secondly}, under the new update scheme, we look for a suitable descent feasible curve along which the objective function can achieve a certain decrease by taking a suitable stepsize. Then the original problem \brev{can be treated} as an unconstrained optimization problem. We consider to combine the efficient BB method \cite{barzilai1988two} with our new update scheme.
To ensure the global convergence, we adopt the adaptive nonmonotone line search \cite{dai2005projected}, leading to an  adaptive feasible Barzilai-Borwein-like (AFBB) method for problem (\ref{prob:Stiefel}). \rev{Note that certain feasible BB-like method with a different nonmonotone line search  was first studied in \cite{wen2013feasible}, but the convergence issue was not discussed there.} \rev{We prove} the global convergence of the AFBB method in the numerical sense under some mild assumptions. Although our update scheme is also a retraction, the convergence of retraction-based line search methods in \cite{absil2008optimization} cannot be applied to our methods.  \rev{To the best of our knowledge, this is the first global convergence result for feasible methods  with nonmonotone line search for optimization on the Stiefel manifold.}  \rev{Furthermore}, our convergence analysis \rev{can also be extended to feasible BB-like methods based on monotone or some other nonmonotone Armijo-type line search techniques.}

\rev{Thirdly, we extend the proposed  update scheme and algorithm to deal with optimization with \rrev{multiple generalized orthogonality constraints}:
\be \label{prob:MultiOrthoStiefel}
    \min_{X_1 \in \mbC^{n_1 \times p_1}, \ldots, X_q \in \mbC^{n_q \times p_q}}\,   \mathcal{F}(X_1, \ldots, X_q),
    \ \ \mbox{s.t.}\ \   X_1^*H_1X_1=K_1,\, \ldots,\, X_q^* H_q^{} X_q = K_q,
\ee
where $H_1\in \mbR^{n_1 \times n_1}, \ldots, H_q \in \mbR^{n_q \times n_q}$ are symmetric positive semidefinite matrices, and $K_1 \in \mbR^{p_1 \times p_1},$ $\ldots,$ $K_q \in \mbR^{p_q \times p_q}$ are symmetric positive definite matrices. Note that problem  \reff{prob:Stiefel} is a special case of problem \reff{prob:MultiOrthoStiefel}. \lastrev{See \cite{lai2014folding,  zhang2012alternating} for two applications of problem \reff{prob:MultiOrthoStiefel}. }}

\rev{Finally, to demonstrate the efficiency of the proposed method, we apply the new method  to a variety of problems. For the nearest low-rank correlation matrix problem, our new method performs remarkably well in terms of speed and solution quality. We also modify the new method to deal with the extra  fixed constraints for the nearest low-rank correlation matrix problem through the augmented Lagrangian function. The preliminary numerical results
show the potential of the new method to handle some more general constraints beyond the sphere constraints.
For the Kohn-Sham total energy minimization problem arising in electronic structure calculations, the new method is considerably competitive with the widely used SCF iteration.}

The rest of this paper is organized as follows. In \S2, we review the existing update schemes and give  the first-order  optimality condition of problem \reff{prob:Stiefel}.  In \S3, we introduce our framework of  constraint preserving update schemes and propose a new update scheme. Some properties of the new update scheme and comparisons with existing update schemes are also stated in this section.  We present the AFBB method in \S\ref{subsection:AFBB}, establish its global convergence in \S\ref{subsection:convergenceAFBB} and then discuss how to deal with the cumulative feasibility error during the iterations  in \S\ref{subsection:error}.  The  AFBB method is also extended to more general problems in \S\ref{section:extension}.  Some numerical tests on an extensive collection of problems are presented to demonstrate the efficiency of our AFBB method in \S\ref{section:numericaltests}. Finally, conclusions are drawn in the last section.

 {\it Notation:} For matrix $M \in \mbR^{n \times n}$, we define $\mbox{sym}(M)=(M+M^{\msT})/2$.
The maximal and minimal eigenvalues of $M$ are denoted by  $\lambda_{\max}(M)$ and  $\lambda_{\min}(M)$, respectively.
We use $\mbox{diag}(M)$ to stand for the vector formed by the diagonal entries of $M$. Meanwhile, we use $\Diag(\theta_1, \ldots, \theta_n)$ to represent the diagonal matrix whose diagonal entries are $\theta_1, \ldots, \theta_n$. The set of $n$-by-$n$ symmetric matrices is denoted by $\mathcal{S}^n$. For $S \in \mathcal{S}^n$, if $S$ is positive semidefinite (positive definite), we mark $S \succeq 0$ ($S \succ 0$).
The Euclidean inner product of two matrices $A, B \in \mbR^{m \times n}$ is defined as $\langle A, B\rangle=\mtr(A^{\msT}B)$, where $\mtr(\cdot)$ is the trace operator. \rev{We denote by $A_{(i)}$ the $i$-th column of $A$}. \rrev{The decomposition $A = \mbox{qr}(A) \mathrm{upp}(A)$ is the unique QR factorization with  $\mbox{qr}(A) \in \mbR^{n \times p}$ being a matrix with orthonormal columns and  $\mathrm{upp}(A) \in \mbR^{p \times p}$ an upper triangular matrix with  positive diagonal entries.}   The condition number of $A$ is defined as $\text{cond}(A) =\Big(\frac{\lambda_{\max}(A^{\msT}A)}{\lambda_{\min}(A^{\msT}A)}\Big)^{1/2}$. Denote by $e_i$ the $i$-th unit vector of an appropriate size. For $X \in \stief $, we define $P_X=I_n - \frac{1}{2}XX^{\msT}$. The gradient of $\mF$ with respect to $X$ is $G\coloneqq\mD \mF(X)=\Big(\frac{\partial \mF(X)}{\partial X_{ij}}\Big)$, whereas the gradient in the  tangent space is denoted by $\nabla \mF$.

\section{Preliminaries} \label{section:prelimanary}
\subsection{Overview of existing update schemes}\label{section:reviewupdateschems}
Given any tangent direction $D \in \mbR^{n \times p}$ satisfying $X^{\msT}D + D^{\msT}X=0$ with $X \in \stief$, we briefly review the geodesic-like and projection-like update schemes. Note that the parameter $\tau \geq 0$ in the following update schemes is some stepsize.

\emph{Geodesic-like update schemes}. In 1998, Edelman {\it et al.} \cite{edelman1998geometry} proposed a computable geodesic update scheme, in which the iterations lie in the curve defined by
\be \label{equ:update:geodesic}
\rrev{Y_{\rev{\mathrm{geo1}}}(\tau; X)  =\big[\ba{l}  X,  \  Q \ea \big]
\exp \left(\tau\left[ \begin{array}{cc}-X^{\msT}D & \ -R^{\msT}\\R & 0\end{array} \right] \right) \left[\begin{array}{c}I_p\\0\end{array}\right],}
    \ee
    \rrev{where $QR =  - (I_n-XX^{\msT})D$ is the unique QR factorization of $-(I_n - XX^{\msT})D$.} This strategy requires computing the exponential of a $2p$-by-$2p$ matrix and the QR factorization of an $n$-by-$p$ matrix at each iteration. Consequently, the flops will be high when $p \geq n/2$. Another geodesic approach is proposed by Abrudan {\it et al.}  \cite{abrudan2008steepest}. Given an $n$-by-$n$ skew-symmetric matrix $\brev{A_1}$,  they considered the curve $Y_{\brev{\mathrm{geo2}}}(\tau) = \brev{\exp}(-\tau \brev{A_1}) X.$
    Comparing with \reff{equ:update:geodesic}, this formula can efficiently deal with the case when $p\geq n/2$.  Nevertheless, it still requires great efforts to compute the exponential of an $n$-by-$n$ matrix.  To avoid computing exponentials of matrices, Nishimori and Akaho \cite{nishimori2005learning} proposed a kind of quasi-geodesic approach.  Given an $n$-by-$n$ skew-symmetric matrix  $\brev{A_2}$, a special case of their update schemes is the Cayley transformation scheme
  \be  \label{equ:update:cayley}
  Y_{\brev{\mathrm{qgc}}}(\tau; X) = \Big(I_n - \frac{\tau}{2} A_2\Big)^{-1} \Big(I_n + \frac{\tau}{2} A_2\Big)X.
  \ee
 The computation cost for \reff{equ:update:cayley} is \rev{$O(n^3)$. This is considerably high even for small $p$.}  In 2010, \rev{by some \rrev{clever} derivations, Wen and Yin \cite{wen2013feasible} developed a simple and efficient constraint preserving update scheme, known as the Crank-Nicholson-like update scheme.} \rev{This scheme is equivalent to the Cayley transformation update scheme}. \rev{Their update scheme is described as follows:}
  \be \label{equ:update:wenyin}
  Y_{\brev{\mathrm{wy}}}(\tau; X)=X-\tau U \Big(I_{2p}+\frac{\tau}{2}V^{\msT}U\Big)^{-1}V^{\msT}X,
  \ee
  where $U=[P_X D, \,X]$, $V=[X,\, -P_XD] \in \mbR^{n \times 2p}$. For convenience, we call it Wen-Yin update scheme throughout  this paper. Formula  \reff{equ:update:wenyin} has the lowest computation complexity per iteration among the existing geodesic-like approaches when $p  < n/2$. However, when $p \geq n/2$, the cost is still expensive. To deal with this case, a low-rank update scheme is explored in \cite{wen2013feasible}. \rev{They were also the first to combine the feasible update scheme with the nonmonotone curvilinear search for optimization with orthogonality constraints. }

 \emph{Projection-like update schemes}. In spite of the nonconvexity of the Stiefel manifold, it is possible to preserve the constraint by the projection.  The projection of a rank $p$ matrix $C \in \mbR^{n \times p}$ onto $\stief$ is defined as the unique solution of
   \be
     \min_{X \in \mathbb{R}^{n \times p}} \, \|X - C\|_{\msF}^{}, \ \ \mbox{s.t.} \ \  X^{\msT}X = I_p,
    \label{equ:projstiefel}
  \ee
  where $\|\cdot\|_{\msF}^{}$ is the Frobenius norm. For any symmetric positive definite matrix $B \in \mbR^{p \times p}$, denote by $B^{1/2}$ the unique \emph{square root} of $B$.  It is easy to see that the solution of \reff{equ:projstiefel} is  $\mP_{\stief}(C) = C(C^{\msT}C)^{-1/2}$. Then we can extend the gradient projection method for optimization with convex constraints for solving \reff{prob:Stiefel}, yielding the update scheme
  \be \label{equ:update:gp}
  Y_{\brev{\mathrm{gp}}}(\tau; X) = \mP_{\stief}(X - \tau G).
  \ee
 In fact, the famous power method \cite{golub1996matrix} for the extreme eigenvalue problem of symmetric matrix is a special case of this gradient projection update scheme.  Manton \cite{manton2002optimization} considered another different projection scheme
  \be
  Y_{\brev{\mathrm{mp}}}(\tau; X) = \mP_{\stief}(X - \tau D). \nn
  \ee
Absil {\it et al.} \cite{absil2008optimization} proposed the  polar decomposition
  \be \label{equ:update:polar}
  Y_{\brev{\mathrm{pd}}}(\tau; X) = (X -\tau D)(I_p + \tau^2 D^{\msT}D)^{-1/2}.
  \ee
  The polar decomposition is equivalent to  Manton's projection update scheme, but has lower complexity cost\footnote{\rev{In \cite{manton2002optimization}, Manton used the SVD decomposition of $X - \tau D$ to obtain the projection, and the cost is higher than that of the polar decomposition.}}. It is then mainly considered in this paper. It is worth noting that the QR factorization update scheme
     \be \label{equ:update:qr}
     Y_{\mathrm{qr}}(\tau; X) = \mbox{qr}(X - \tau D)
    \ee
  proposed in \cite{absil2008optimization} can be regarded as an approximation to the polar decomposition update scheme.
 \subsection{First-order optimality condition}
To begin with, we introduce some basic concepts related to the Stiefel manifold as in \cite{edelman1998geometry}. For any $X \in \stief$,  define the tangent space at $X$ as $\mT_X\coloneqq\{\Delta : X^{\msT}\Delta+\Delta^{\msT}X=0\}$.
There are two different metrics on $\mT_X$. The first one is the \emph{Euclidean metric} $d_e(\Delta,\Delta)=\langle \Delta, \Delta\rangle, $ which is independent of the point $X$. The second one is the \emph{canonical metric} $d_c(\Delta,\Delta)=\langle \Delta, P_X\Delta\rangle,$ which is related to $X$.  \rrev{Then  the projection of any $Z \in \mbR^{n \times p}$ onto the tangent space $\mT_X$ under the Euclidean or canonical metric is  $\proj(Z)= Z - X \mbox{sym}(X^{\msT}Z)$.}

The gradient $\nabla \mF \in \mT_X$ of a differential function $\mF(X): \mbR^{n \times p} \rightarrow \mbR$ on the Stiefel manifold is defined such that
\be \label{equ:partialStiefel}
  \langle G, \Delta \rangle =d_c(\nabla \mF, \Delta)\equiv \langle \nabla \mF , P_X \Delta\rangle
\ee
for all tangent vectors $\Delta$ at $X$. Solving (\ref{equ:partialStiefel}) for $\nabla \mF$ such that $X^{\msT}\nabla \mF$ is skew-symmetric yields
 $$\nabla \mF=G-XG^{\msT}X.$$ Notice that $\nabla \mF$ is not the projection of $G$ onto the tangent space at $X$. The latter should be $G-X\mbox{sym}(X^{\msT}G).$

We  now give the first-order optimality condition without proof. It is analogous to Lemma 2.1 in \cite{wen2013feasible}.
\begin{lemma}[First-order optimality condition]\label{lemma:optimality}
 Suppose $X$ is a local minimizer of problem (\ref{prob:Stiefel}). Then $X$ satisfies the first-order optimality condition
$$  \mathcal{D}_X\mL(X,\Lambda)=G - XG^{\msT} X =0;$$
{\it i.e.,} $\nF=0$, with the associated Lagrange multiplier $\Lambda=G^{\msT} X$. Besides, \rev{$\nabla \mF = 0$ if and only if}
$$  G - X\big(2\rho G^{\msT}X + (1-2\rho)X^{\msT}G\big) =0, \quad\mbox{for any $\rho >0$}. $$
\end{lemma}
\section{Constraint Preserving update schemes}\label{section:update}

For a feasible point $X \in \stief,$ denote $\mR_X = \{XR: R \in \mbR^{p \times p}\}$ and $\mN_X = \{W \in \mbR^{n \times p}: X^{\msT}W = 0\}$ to be the range space of $X$ and the null space of $X^{\msT}$, respectively. It is well known that the two spaces are orthogonal to each other and their sum forms the whole space $\mbR^{n\times p}$. As a result, any point in $\mbR^{n\times p}$ can uniquely be decomposed into the sum of two points, which belong to the two spaces, respectively. With this observation, we introduce our idea for a framework of constraint preserving update schemes for problem \reff{prob:Stiefel}.

Given a matrix $W$ in the null space $\mN_X$, consider the following curve
\be \label{equ:update1}
  Y(\tau; X) = X R(\tau) + W N(\tau),
\ee
where
$R(\tau), N(\tau) \in \mbR^{p \times p}$ and $\tau\ge 0$ is some parameter. In other words, this curve can be divided into two parts; {\it i.e.}, $XR(\tau)$  in the range space of $X$ and $W N(\tau)$ in the null space of $X^{\msT}.$
\revmini{To make the following analysis simple,  we assume that} $R(\tau) = \wR(\tau) Z(\tau)$ and $N(\tau) = \wN(\tau) Z(\tau)$, where $Z(\tau)$ is always invertible, the curve \reff{equ:update1} can be expressed by
\be
 Y(\tau; X) = \left(X \wR(\tau) + W \wN(\tau)\right) Z(\tau). \label{add1}
\ee

Our goal is to determine appropriate \revmini{$\wR(\tau)$, $\wN(\tau)$, $Z(\tau)$ and $W$}  such that the curve $Y(\tau; X)$ is still feasible; {\it i.e.}, $Y^{\msT}Y = I_p$ (notice that if there is no confusion,  we will write $Y(\tau; X)$ as $Y(\tau)$ or even simply $Y$, {\it etc.}).
\revmini{To do so, we need to investigate the fundamental relations which hold for $Y(\tau)$. This technique can be called as {\it method of undetermined coefficients}.}

\revmini{Firstly, a natural and necessary  condition is  that $Y^{\msT} Y = I_p$, which requires}
\be
Z(\tau)^{\msT}\left(\wR(\tau)^{\msT} \wR(\tau) + \wN(\tau)^{\msT} W^{\msT}W \wN(\tau)\right)Z(\tau) = I_p, \nn
\ee
or, equivalently,
 \be \label{equ:ztau2}
 Z(\tau)^{-\msT}Z(\tau)^{-1} = \wR(\tau)^{\msT} \wR(\tau) + \wN(\tau)^{\msT} W^{\msT}W \wN(\tau).
 \ee


 \revmini{Secondly, consider some initial conditions which $Y(\tau; X)$ should satisfy.}  As $Y(\tau; X)$ with $\tau\ge 0$ is a curve starting from
the current iteration $X$, it is natural to impose that $Y(0; X) = X$. To meet this condition, by \reff{add1}, we may ask
\be \label{equ:condition1}
\wR(0)=I_p, \quad \wN(0)=0,\quad Z(0) = I_p.
\ee
With these choices, we further have by \reff{add1} that
\be Y'(0; X) = X\left(\wR'(0) + Z'(0)\right) + W\wN'(0). \label{add2}\ee
Assume that some matrix $E$ is chosen, which is intended for $-Y'(0; X)$; {\it i.e.}, $E=-Y'(0; X)$. Then \reff{add2} holds if
\be \label{equ:condition2}
W = -(I_{n} - XX^{\msT})E, \quad  \wR'(0)+ Z'(0) = -X^{\msT}E,\quad  \wN'(0) =I_p.
\ee

\rrev{In the following, we consider three approaches of choosing \revmini{$W$}, $\wR(\tau)$, $\wN(\tau)$  and $Z(\tau)$, \revmini{which satisfy the requirements  \reff{equ:ztau2}, \reff{equ:condition1} and \reff{equ:condition2},} such that \reff{equ:ztau2} holds.}

\emph{Approach \rrev{I}}. Consider the simple case that $Z(\tau) \equiv I_p$. This, with the second equation in \reff{equ:condition2}, indicates that $\wR'(0) = -X^{\msT} E$. \revmini{By some tedious analysis in Appendix \ref{appendix:subsection:approach1}, we can obtain a generalized geodesic update scheme which covers the geodesic update scheme \reff{equ:update:geodesic} as a special case.}

\rrev{In the next two approaches, to} satisfy the conditions $\wN(0)=0$ and $\wN'(0) =I_p$, we choose
\be \wN(\tau) = \tau I_p. \label{add3} \ee
 \rrev{Noticing that $\frac{\mathrm{\partial}Z(\tau)^{-1}}{\mathrm{\partial} \tau}  = - Z(\tau)^{-1}Z'(\tau) Z(\tau)^{-1}$} and using \reff{equ:ztau2} and \reff{add3}, we can get that
\be Z'(0)^{\msT} + Z'(0) = - \left(\wR'(0)^{\msT} + \wR'(0)\right).
\nn
\ee
This, with the second condition
in \reff{equ:condition2}, means that $X^{\msT}E$ must be a skew-symmetric matrix; {\it i.e.}, $E \in \mT_X.$

\emph{Approach \rrev{II}}.\ \revmini{To meet $\wR(0) = I_p$}, we consider to choose
\be \label{equ:approach2:wR}
\wR(\tau) = I_p + \tau \wR'(0).
\ee
\revmini{Noting that}  $\wN(\tau) = \tau I_p$, we can get $Z(\tau)$ from \reff{equ:ztau2} by the  polar decomposition or the Cholesky factorization. \revmini{Thus we can obtain the generalized polar decomposition and Cholesky factorization update schemes. The former one covers the ordinary polar decomposition and gradient projection update schemes, while the latter one includes the ordinary QR factorization update scheme. See Appendix \ref{appendix:subsection:approach2} for details.}

\emph{Approach \rrev{III}}. In this approach, \revmini{to solve  $Z(\tau)$ from \reff{equ:ztau2} easily, we may assume $\wR(\tau)$ to be  some function of $Z(\tau)$, which takes the form of }
\be \label{equ:wR1}
  \wR(\tau) = 2I_p - Z(\tau)^{-1}.
\ee
Substituting  \reff{add3} and \reff{equ:wR1} into \reff{equ:ztau2} leads to
$$  Z(\tau)^{-\msT} + Z(\tau)^{-1} = 2I_p + \frac{\tau^2}{2} W^{\msT}W. $$
Consequently, $Z(\tau)$ must be of the form
\be \label{ztau1}
  Z(\tau) = \Big( I_p + \frac{\tau^2} 4 W^{\msT}W + L(\tau)\Big)^{-1},
\ee
where $L(\tau)$ is  any  $p$-by-$p$ skew-symmetric matrix with $L(0) = 0$.  Notice that the above inverse always exists since $L(\tau)$ is skew-symmetric. The relations \reff{equ:wR1} and \reff{ztau1} indicate that $\wR'(0) = Z'(0) = -L'(0)$. Further, by the second relation in \reff{equ:condition2}, we must have that $L'(0) = \dfrac{1}{2}X^{\msT}E$. Thus we can choose
\be L(\tau) = g(\tau)X^{\msT}E, \nn\ee
where $g(\tau)$ is any function  satisfying $g(0) = 0$ and  $g'(0) =1/2$. \rev{For simplicity, we choose $g(\tau) = \tau/2$.  See \S\ref{subsection:convergenceAFBB} and \S\ref{subsection:hetequadratic} for more choices of $g(\tau)$.}

To sum up, given any matrix $E \in \mT_X,$ we can define the following update scheme
 \be\label{equ:update:new}
 \left\{
 \begin{array}{l}
   W = -(I_n - XX^{\msT})E,\\
   J(\tau) = I_p + \frac{\tau^2}{4}W^{\msT}W +  \rev{\frac{\tau}{2}} X^{\msT}E, \\
   Y(\tau; X) = (2X + \tau W) J(\tau)^{-1} - X.
 \end{array}
 \right.
 \ee
\rev{Some geometrical meanings of the above scheme
   will be discussed in the paragraph before Lemma \ref{lemma:projection}.}

   A few remarks on the framework and the direction $E$ are made here. Firstly, it follows from \rrev{\reff{equ:update:geodesic:frame2}}, \reff{equ:generalpolar} and \reff{equ:generalqr} that  the framework can unify the famous \rrev{geodesic,} \rev{gradient projection},  polar decomposition and QR factorization update schemes. Meanwhile, it can yield  \rrev{generalized geodesic}, polar decomposition or QR factorization update \brev{scheme} by choosing different $\rev{\wF}$ or  $Z'(0)$. \revmini{See Appendix \ref{appendix:subsection:approach1} and \ref{appendix:subsection:approach2}, respectively, for the specific choices of $\wF$ and $Z'(0)$.} We will mainly consider the new update scheme \reff{equ:update:new} in the remainder of  this paper. Secondly, like  unconstrained optimization, \rev{many possible choices for $E$}, \rev{for instance},  the gradient descent, conjugate gradient, or quasi-Newton direction, can be used under the update scheme \reff{equ:update:new}. \brev{However}, we \brev{focus on} the gradient descent direction in \S\ref{section:choosingE}
 due to its simplicity.

\subsection{Choice of $E$}\label{section:choosingE}
In this subsection, we consider to seek an appropriate $E$ such that the update scheme $Y(\tau; X)$ given by \reff{equ:update:new} defines a descent curve. \revmini{We summarize the properties of the update scheme \reff{equ:update:new} in the following lemma. See Appendix \ref{appendix:section:proof1}} for its proof.

 \begin{lemma} \label{lemma:update}
   For any feasible point $X \in \stief$, consider the curve given by \reff{equ:update:new}, where $E = \Drho$ and
 \be\label{equ:Drho:def}
 \Drho= (I_n - (1 - 2\rho)XX^{\msT})\nF =  G - X\big(2\rho G^{\msT}X + (1-2\rho)X^{\msT}G\big), \quad \rho>0.   \ee
 Then the following properties hold:
 \begin{enumerate}
   \item[\emph{(i)}] $Y(\tau)^{\msT} Y(\tau)= I_p;$
   \item[\emph{(ii)}] $Y'(0) = -\Drho$ is a descent direction and
      \be
   \mF'_{\tau}(Y(0))\coloneqq\frac{\partial \mF(Y(\tau))}{\partial \tau}\bigg{|}_{\tau=0} \leq -\min\{\rho, 1\}\|\nF\|_{\msF}^2;  \nn
       \ee
     \item[\emph{(iii)}] for any $\tau>0$ and \rev{any $p$-by-$p$ orthogonal matrix \rrev{$Q_p$}}, we have  $Y\rrev{Q_p} \ne X;$
    \item[\emph{(iv)}]  \rev{$\mathrm{cond}(J) \leq (5 + \upsilon^2)/4$,  where $\upsilon = \tau \|\Drho\|_{\msF}^{}$.}
 \end{enumerate}
 In particular,  if $p=n$, $Y(\tau) = X(2J^{-1} - I_n)$ \rev{and $J = I_n + \frac{\tau}{2} X^{\msT} E$}; if $p=1$, the matrices $X$, $Y$, $G$ reduce to the vectors $x$, $y$, $g$, respectively, and the update scheme becomes
 $$y(\tau) = \bigg(\frac{2 + \tau x^{\msT}g }{1 + \frac{\tau^2}{4}(g^{\msT}g - (x^{\msT}g)^2)} - 1 \bigg) x - \frac{\tau}{1 + \frac{\tau^2}{4}(g^{\msT}g - (x^{\msT}g)^2)}g.$$
\end{lemma}

Before we proceed, several remarks on  Lemma \ref{lemma:update} are in order. Firstly, \rev{without otherwise specification, we will always choose $E= \Drho$ in the remainder of this paper. In this case, by  \reff{equ:update:new} and the definition of $\Drho$, we must have that}
   \be
   \rev{W = -(I_n - XX^{\msT})\Drho \equiv  -(I_n - XX^{\msT})G.} \label{equ:W:D:G}
 \ee
There are two special choices of $\rho$. One is $\rho = 1/2$, yielding $Y'(0) = -\nF$; the other one is $\rho = 1/4$ yielding $Y'(0) = -G + X \mbox{sym}(X^{\msT}G)$.
\rrev{Generally, the two directions are not the same except the case when $X^{\msT}G = G^{\msT}X$.  We will compare the two directions in \S \ref{subsection:hetequadratic}; the results therein shows that $\rho = 1/4$ is a better choice.}
 Secondly, recalling that the Grassmann manifold $\mathcal{G}_{n,p}$ is defined as the set of all $p$-dimensional subspaces of an $n$-dimensional space, any two orthogonal matrices whose columns span the same $p$-dimensional subspace can be regarded  as the same point in $\mathcal{G}_{n,p}$. By (iii) of  Lemma \ref{lemma:update}, we know that any $Y(\tau)$ with $\tau>0$ and $X$ must be different points in $\mathcal{G}_{n,p}$. Thus our update scheme can be also used for optimization on the Grassmann manifold, \rev{that is,   problem \reff{prob:Stiefel}  with the additional \rrev{homogeneity} assumption that $\mF(X\rrev{Q_p}) = \mF(X)$, where $\rrev{Q_p}$ is any $p$-by-$p$ orthogonal matrix.}
 Finally, statement (iv) of Lemma \ref{lemma:update} indicates that the condition number of $J$ can be controlled by the term \rev{$\tau \|\Drho\|_{\msF}$, this fact will guide us to set a safeguard for the stepsize, as addressed in \reff{equ:algorithm:safeguard}.}

 \rev{As mentioned before, one typical choice of $E$ is} $E=D_{\rho}$. \rev{Consider the case when $p \approx n$,} it is quite expensive to calculate $J^{-1}$ directly. However, in this case we may construct a low-rank matrix $E$ so that $J^{-1}$ can be cheaply obtained. \revmini{Lemma \ref{lemma:lowrank}}  shows the possibility of choosing a rank-$2$ matrix $E$, with which $J^{-1}$ can be analytically given and fast computed. \revmini{A proof of Lemma \ref{lemma:lowrank} can be found in Appendix \ref{appendix:section:proof2}.} Similarly, we may also form a rank-$2r$ matrix $E$ with $1\leq r < n/2$ in the same vein.  \rev{Nevertheless, it is worth noting that seeking an appropriate low-rank matrix $E$ faces a trade-off between the computational cost and the quality of the search curve.}
\begin{lemma}\label{lemma:lowrank}
    For any feasible $X \in \stief$,  define $D^{(i)} = G_{(i)}^{} e_i^{\msT} - X_{(i)}^{} G_{(i)}^{T} X$.
    Consider the curve $Y(\tau)$  given by \reff{equ:update:new} with  $E = D^{(q)}$, where $q$ is the column index given by
\be \label{equ:lowrank_q}
q\coloneqq \argmax\limits_{i=1,\ldots, p}\  \langle G,  D^{(i)} \rangle =
\rrev{\argmax\limits_{i=1,\ldots, p}\  e_i^{\msT}\left(G^{\msT} \nabla \mF\right) e_i.}
\ee
	 Then we have that
\be\label{equ:lowrank:inverseJ}
\rrev{J^{-1}  = I_p - \frac{1}{1 + \alpha} \big[\ba{cc} e_q,&\  b \ea \big]
\left[\ba{cc} \alpha &\  -1 \\ 1  & \ 1\ea \right] \left[\ba{c}e_q^{\msT}\\[2pt]b^{\msT}\ea\right]},
\ee
where $\alpha =  \frac{\tau^2}{4} G_{(q)}^{T} \left( I_n - X_{(q)}^{} X_{(q)}^{\msT} \right) G_{(q)}^{}$,
\rrev{$b = \frac{\tau}{2} X_{-q}^{\msT} G_{(q)}$ and $X_{-q} = X - X_{(q)}^{}e_q^{\msT}$.}
 Moreover, $Y(\tau)$ is a descent curve satisfying
 $$\mF'_{\tau}(Y(0)) \leq -\frac{1}{2p} \|\nabla \mF\|_{\msF}^2.$$
  \end{lemma}

{\color{black}Geometrically, the proposed framework is not geodesic expect for the case when  $p = 1$.
Instead, it can be regarded as a generalized gradient projection scheme. In the special case when there always holds $X^{\msT}G \equiv G^{\msT}X$, we can explicitly get the projection operator  of new update scheme \reff{equ:update:new}. Notice that the condition $X^{\msT} G \equiv G^{\msT}$X holds for a wide range of problems, such as the linear eigenvalue problem and the vector case ($p= 1$)  of problem \reff{prob:Stiefel}.}

\begin{lemma} \label{lemma:projection}
  Assume that $X^{\msT}G \equiv G^{\msT} X$ for any feasible point $X \in \stief$. Then the update scheme \reff{equ:update:new} with $E = \Drho$ can be expressed by
$$Y(\tau) = \mP_{\stief}\left(X\Big(I_p + \tau X^{\msT}G -\frac{\tau^2}{4}G^{\msT}(I_n - XX^{\msT})G\Big) -\tau G\right). $$
\end{lemma}
\begin{proof}
  The condition $X^{\msT}G \equiv G^{\msT} X$ implies that $X^{\msT} \Drho = 0$, \rev{which with \reff{equ:update:new} means that} $J = I_p + \frac{\tau^2}{4}W^{\msT}W$. Rewrite $Y(\tau)$ in \reff{equ:update:new} as $$Y(\tau) =\big(X(2I_p-J) + \tau W\big)J^{-1}.$$Thus it follows from $Y(\tau)^{\msT}Y(\tau) = I_p$  and  \rev{the above expressions of $Y(\tau)$ and $J$} that
$$\big(X(2I_p-J) + \tau W\big)^{\msT} \big(X(2I_p-J) + \tau W\big) = J^2.$$
\rev{Recalling the definition of $\mP_{\stief}(\cdot)$ for \reff{equ:projstiefel}, we have that}
\begin{align}
Y(\tau)&= \mP_{\stief} \big(X(2I_p-J) + \tau W\big) \nn \\
&= \mP_{\stief}\left(X\Big(I_p + \tau X^{\msT}G -\frac{\tau^2}{4}G^{\msT}(I_n - XX^{\msT})G\Big) -\tau G\right),\nn
 \end{align}
 where the second equality uses \reff{equ:W:D:G}.  This completes the proof.
 $\smartqed$
\end{proof}

\subsection{Comparison with the existing update schemes}
To begin with, we address a relationship between the update scheme \reff{equ:update:new} and the Wen-Yin update scheme. We show that
the Wen-Yin update scheme can be regarded as a special member of the update scheme \reff{equ:update:new}
with $g(\tau) = \tau/2$ under the condition that $I_p + \frac{\tau}{4}X^{\msT}D$ is invertible. Notice that this invertible condition is  indispensable for the well definition of the Wen-Yin update scheme, but is not necessary for the validity of the update scheme \reff{equ:update:new}.
\revmini{We summarize the results in Proposition \ref{proposition:update_equ}  and relegate the proof in Appendix \ref{appendix:section:proof3}.}

\begin{proposition} \label{proposition:update_equ}
  For any feasible point $X \in \stief$, if the tangent direction  $D \in \mathcal{T}_X$ is such that $I_p+\frac{\tau}{4}X^{\msT}D$ is invertible, \lastrev{then} the  update scheme \reff{equ:update:wenyin} is well-defined and it is equivalent to  the update scheme \reff{equ:update:new}  with $g(\tau) = \tau/2$ and $E = D$.
\end{proposition}

\revmini{For a fixed feasible point $X \in \stief$, assume that the gradient $G$ has been computed. Let us compare the computational costs of several aforementioned constraint preserving update schemes.
  In  Table \ref{table:cost}, we list the cost of the aforementioned update schemes and leave some details in Appendix \ref{appendix:section:details}.}  In the table, the column ``first $\tau$'' gives the cost of computing a feasible point with the first trial stepsize of \brev{$\tau_{\mathrm{first}}$}, whereas the column ``new $\tau$'' provides the cost of getting a new feasible point with a new stepsize of \brev{$\tau_{\mathrm{new}}$}.  Table \ref{table:cost} tells us that the ordinary gradient projection actually has the lowest complexity cost while the scheme \reff{equ:update:qr} based on the QR factorization and our update scheme  \reff{equ:update:new} are strong candidates. Although by Proposition \ref{proposition:update_equ}, our update scheme and the Wen-Yin update scheme are equivalent under some assumption, the former has a \lastrev{lower} complexity cost especially for large $p$, as shown in Table \ref{table:cost}.  As pointed in \cite{absil2012projection}, however, the choice of the update scheme can affect the number of iterations required for solving the optimization problem. Hence, a lower complexity cost at each iteration does not \brev{necessarily} imply a higher efficiency on the whole. Actually, how to seek a constraint preserving update scheme which can find the global minimizer with \brev{a higher} probability and at \brev{a faster} speed is still an open problem. This remains under investigation.

 \begin{table}[!htbp]
 \caption{Computational cost for different update schemes} \label{table:cost}
\centering
~\\[2pt]
\linespread{1.9}
\revmini{
\begin{small}
  \begin{tabular}{@{}r@{,\,}l@{\hspace{3mm}}r@{\,+\,}r@{\,+\,}r@{\hspace{4.5mm}}  r@{\,+\,}r@{\,+\,}r   @{\hspace{0.3mm}}c@{\hspace{4.2mm}}c@{\hspace{3.0mm}}c@{\hspace{0.3mm}}c@{\hspace{4.2mm}}r@{\hspace{3.0mm}}r@{}}
  \Xhline{0.6pt}
  \multicolumn{2}{c}{\multirow{2}{*}{update schemes}} & \multicolumn{6}{c}{$1<p<n$}&& \multicolumn{2}{c}{$p=1$}&& \multicolumn{2}{c}{$p=n$} \\
  \Xcline{3-8}{0.6pt} \Xcline{10-11}{0.6pt} \Xcline{13-14}{0.6pt}
  \multicolumn{2}{c}{}    & \multicolumn{3}{c}{first $\tau$} & \multicolumn{3}{c}{new $\tau$}& &first $\tau$ &new $\tau$ & &first $\tau$ &new $\tau$  \\  \Xhline{1.1pt}
 $^{^{\phantom{d}}}$ $Y_{\mathrm{gp}} \reff{equ:update:gp}$   & any $D$  & $3np^2$ &   $2np$ & $\frac{32}{3}p^{3}$  & $3np^2$&  $2np$ & $\frac{32}{3}p^3$  && $5n$ & $5n$ &&$\frac{41}{3}n^3$ & $\frac{41}{3}n^3$\\
 $Y_{\mathrm{geo1}}\reff{equ:update:geodesic}$ &$D = \Drho$  & $10np^2$ &$np$ & $80 p^3$ & $4np^2$ &  & $80p^3$ & &$7n$ & $3n$ && $14n^3$ & $12n^3$\\
 $Y_{\mathrm{pd}} \reff{equ:update:polar}$& $D = \Drho$ & $7np^2$&   $3np$ &  $\frac{32}{3}p^3$ & $2np^2$ & $2np$ &$\frac{32}{3}p^3$ & & $7n$ & $3n$&& $\frac{53}{3}n^3$ & $\frac{38}{3}n^3$\\
 $Y_{\mathrm{qr}}\reff{equ:update:qr}$& $D = \Drho$     & $6np^2$&  \multicolumn{2}{l}{\hspace{-2mm}$3np$}    &  $2np^2$ & \multicolumn{2}{l}{\hspace{-2mm}$2np$} & & $7n$ & $3n$ && $6n^3$ & $2n^3$\\
$Y_{\mathrm{wy}} \reff{equ:update:wenyin}$& $D = \Drho$& $9np^2$&   $2np$ & $\frac{40}{3}p^3$& $4np^2$ & $np$ & $\frac{40}{3}p^3$& & $7n$ & $3n$ && $\frac{67}{3}n^3$ & $\frac{52}{3}n^3$ \\
 $Y_{\mathrm{wy}} \reff{equ:update:wenyin}$& $D = D_{1/2}$& $7np^2$&  $np$ & $\frac{40}{3}p^3$ & $4np^2$ &  $np$ & $\frac{40}{3}p^3$&&  $7n$ & $3n$ && $\frac{61}{3}n^3$ & $\frac{52}{3}n^3$\\
$Y \reff{equ:update:new}$ & $D = \Drho$ & $7np^2$&  $2np$ & $\frac{5}{3}p^3$ &$2np^2$ & $3np$ & $\frac{2}{3}p^3$ & &$7n$ &  $3n$ && $\frac{14}{3}n^3$ & $\frac{8}{3}n^3$ \\
 \Xhline{0.6pt}
\end{tabular}
\end{small}
}
\end{table}

\section{Adaptive feasible BB-like (AFBB) method and global convergence}\label{section:algorithm}
In this section, we focus on  the adaptive feasible BB-like  method and  its global convergence. We also propose a strategy to control the feasibility error in practical computations.

\subsection{Adaptive feasible BB-like method}\label{subsection:AFBB}
To provide an efficient scheme, we must also pay much attention on choosing the stepsize $\tau$ in the constraint preserving update scheme \reff{equ:update:new} with $E=\Drho$.  Although there is only one parameter, its choice proves very important to the efficiency of the scheme.  Since BB-like methods need less storage and are very efficient for unconstrained optimization problems \cite{dai2006cyclic, fletcher2005barzilai,raydan1997barzilai, zhou2006gradient} and special constrained optimization problems \cite{birgin2000nonmonotone,dai2005projected,dai2006new, wen2013feasible}, we consider to use some BB-like stepsize in the scheme \reff{equ:update:new}.

Denote $\brev{\mathrm{S}}_{k-1}=X_k-X_{k-1}$ and $\brev{\mathrm{Y}}_{k-1}=D_{\rho,k}-D_{\rho,k-1},$ where $X_k = Y(\tau_{k-1}; X_{k-1})$. Similar to the unconstrained case, we can get the large and short BB stepsizes as follows:
  \be
  \tau_k^{\brev{\mathrm{LBB}}}=\frac{\langle \brev{\mathrm{S}}_{k-1}, \brev{\mathrm{S}}_{k-1}\rangle }{|\langle \brev{\mathrm{S}}_{k-1}, \brev{\mathrm{Y}}_{k-1}\rangle |}
  \ \ \text{and}\ \ \tau_k^{\brev{\mathrm{SBB}}}=\frac{|\langle \brev{\mathrm{S}}_{k-1},\brev{\mathrm{Y}}_{k-1}\rangle |}{\langle \brev{\mathrm{Y}}_{k-1}, \brev{\mathrm{Y}}_{k-1}\rangle}. \nn
\ee
In the numerical experiments in \S\ref{section:numericaltests}, we adopt the following alternative BB (ABB) stepsize   \cite{dai2005projected}
\be \label{equ:Stepsize:ABB}
 \tau_{k}^{\brev{\mathrm{ABB}}}= \left\{ \begin{array}{ll}\tau_k^{\brev{\mathrm{SBB}}}, & \mbox{~for odd~} k; \\ \tau_k^{\brev{\mathrm{LBB}}}, & \mbox{~for even~}  k.  \end{array}\right.
\ee
If $J_{k-1}^{-1}$ is  stored in updating $X_k$, we can \rev{freely} get $\langle \brev{\mathrm{S}}_{k-1}, \brev{\mathrm{S}}_{k-1} \rangle =4p-4\mtr(J_{k-1}^{-1})$ \rev{which is due to  the feasibility of $X_k$ and $X_{k-1}$ \rrev{and $X_{k-1}^{\msT} X_k^{} = 2J_{k-1}^{-1} - I_p$}}. Thus, computing the ABB stepsize needs  at most $6np$ flops.

Although BB-like methods prove efficient for nonlinear optimization, its heavy nonmonotonicity in function values makes the global convergence analysis difficult. Up to now, the unmodified BB method is only showed to be globally convergent in the \rev{strongly}  convex quadratic case \cite{raydan1993barzilai} and the convergence is $R$-linear \cite{dai2002convergence}.   It is shown in \cite{Daijiang2012} that the unmodified BB method needs not converge for \brev{the} extreme eigenvalue problem, which is a special case of problem \reff{prob:Stiefel}. \brev{Thus} we consider to incorporate the ABB stepsize with the \rev{Armijo-type line search which requires $\tau_k$ to satisfy}
\be 
\mF(Y_k(\tau_k))  \leq \mF_r + \delta \tau_k \mF_{\tau}'\left(Y(0;X_k)\right),\nn
\ee
\rev{for some constant $\delta \in (0,1)$, and} $\mF_r$ is the so-called ``reference function value'', satisfying $\mF_r \geq \mF_k$.
\rev{In this paper, we consider to update $\mF_r$ by the adaptive strategy proposed in \cite{dai2005projected}. This strategy \lastrev{cannot} only guarantee the convergence, but can keep the efficiency of the unmodified BB method since only few line search procedures will be invoked during the iterations. Denote by $\mF_{\rev{\mathrm{best}}}$ the current best function value and by $\mF_c$ the maximum objective value since the value of $\mF_{\rev{\mathrm{best}}}$ was found.  Initially, we \rrev{set} $\mF_r \coloneqq + \infty,\, \mF_{\rev{\mathrm{best}}} = \mF_{c} = \mF(X_0)$.  Let $L$ \rrev{be} a preselected positive integer.  The following is a detailed description of the aforementioned adaptive strategy}.
\be \label{equ:updateFr}
\ba{ll}
\mbox{if}& \mF_{k+1} < \mF_{\brev{\mathrm{best}}} \\
&\mF_{\brev{\mathrm{best}}} = \mF_{k+1},\  \mF_c = \mF_{k+1},\  l =0,\\
\mbox{else}& \\
& \mF_c = \max\{\mF_c,\mF_{k+1}\},\  l = l+1,\\
&\mbox{if} \  l = L, \ \ \mF_r = \mF_c,  \ \mF_c = \mF_{k+1}, \ l = 0, \ \ \mbox{end} \\
\mbox{end}
\ea
\ee
\rev{For other nonmonotone line search methods, interested readers can refer to \cite{dai2001adaptive, grippo1986nonmonotone, raydan1997barzilai, toint1997non, zhang2004nonmonotone}.}

\rev{Now we give a detailed description of the AFBB method.}
 \begin{algorithm} \label{algorithm:afbb}
(Adaptive Feasible BB-like   Method)
\begin{description}
\item[Step 0] 
  \brev{Given $X_0 \in \stief$,  $\epsilon, \rho>0$, $0<\sigma, \delta<1$, $\epsilon_{\min}, \epsilon_{\max} >0,   \Delta, L > 0$ and set $k\coloneqq0$.}
\item[Step 1] If $\|D_{\rho,k}\|_{\msF}^{} \leq \epsilon,$ stop.
\item[Step 2] Find the least nonnegative integer $i_k$ satisfying
               \be \label{equ:linesearch:Armijo}
	       \mF(Y(\delta^{i_k}\tau_k^{(1)}; X_k)) \leq  \mF_r + \delta\sigma^{i_k}\tau_k^{(1)} \mF'_{\tau}(Y(0; X_k)),
               \ee
	       and set $\tau_k \coloneqq \sigma^{i_k}\tau_k^{(1)}$.
\item[Step 3]
      Update  \rev{$X_{k+1}=Y(\tau_k; X_k)$ by \reff{equ:update:new} and $\mF_r$ by \reff{equ:updateFr}.}
\item[Step 4] Calculate $\tau_k^{(0)}$ by  a certain  BB stepsize and set
                  \be \label{equ:algorithm:safeguard}
		  \tau_k^{(1)} \coloneqq \max \Big\{\epsilon_{\min}/\|D_{\rho,k}\|_{\msF}^{}, \min\big\{\tau_k^{(0)}, \min\{\epsilon_{\max}/\|D_{\rho,k}\|_{\msF}^{},\Delta\}\big\}\Big\}.
                 \ee
\item[Step 5] $k\coloneqq k+1$.  Go to Step 1.
\end{description}
\end{algorithm}
\subsection{Convergence of the AFBB method}\label{subsection:convergenceAFBB}
In this subsection, we establish the global convergence result of Algorithm \ref{algorithm:afbb} in the numerical sense.
\begin{theorem}\label{theorem:convergence}
 Let $\{X_k: k\geq 0\}$ be the sequence generated by Algorithm \ref{algorithm:afbb} with $\epsilon =0$.
 \rev{Assume that $\mF(X)$ is differentiable and that its gradient $\mD \mF(X)$ is Lipschitz continuous on $\stief$ with Lipschitz constant $L_0$, that is,}
 $$\rev{\|\mD \mF(X) - \mD \mF (Y)\|_{\msF}^{} \leq L_0 \|X- Y\|_{\msF}^{}, \quad \mbox{for all $ X, Y \in \stief. $} }$$
Then we have either $D_{\rho,k} =0$ for some finite $k$ or
\be
  \liminf_{k \rightarrow \infty} \|D_{\rho,k}\|_{\msF}^{} = 0. \label{add11}
\ee
\end{theorem}
\begin{proof} \rev{The sketch of proof is as follows (see \cite{jiang2012framework} for more details). Firstly, for the update scheme \reff{equ:update:new}, we can show (see Lemma 4.4 of \cite{jiang2012framework}) that}
  \be\label{equ:theorem:convergence:well}
\rev{\|Y_k(t_k) - Y_k(0)\|_{\msF}^{} \leq \frac{2 + \epsilon_{\max}}{2}   t_k \|\Drhok\|_{\msF}^{}, \quad
\|Y_k'(t_k) - Y_k'(0)\|_{\msF}^{}  \leq  \frac{4 + \epsilon_{\max}}{2}  t_k \|\Drhok\|_{\msF}^2,}
   \ee
\rrev{where $Y_k(t_k)$ and $Y_k'(t_k)$ stand for $Y(t_k; X_k)$ and $Y'(t_k; X_k)$, respectively. }

Secondly, \rev{since $\mD \mF(X)$ is Lipstchitz continuous on the compact set $\stief$, there exists a constant $c_0 > 0$ such that}
$$\rev{\|\mD \mF(X) \|_{\msF}^{} \leq \rev{c_0}, \quad \mbox{for all $X \in \stief$}.} $$
\rev{By employing \reff{equ:theorem:convergence:well}, we can verify \cite{jiang2012framework} that the stepsizes $\{\tau_k\}$ are bounded below, as addressed by}
 \be\label{equ:theorem:convergence:lowbound}
 \rev{ \tau_k \geq \min\{c,\tau_k^{(1)}\},}
 \nn
 \ee
 \rev{where $c = \frac{2\sigma(1-\delta)\min\{\rho,1\}}{c_1  \max\{2\rho,1\}^2}$ and $\ c_1 = \frac{4 + \epsilon_{\max}}{2}c_0 + \frac{2 + \epsilon_{\max}}{2}L_0$.}

Thirdly, \brev{denote} $\mF_r, \mF_c, \mF_{\mathrm{\brev{best}}}, l$ at the $k$-th iteration by $\mF_r^k, \mF_c^k, \mF_{\mathrm{\brev{best}}}^k, l^k$, respectively. It follows from  \reff{equ:linesearch:Armijo},  the second statement of Lemma \ref{lemma:update}, $\tau_k \geq \min\{c,\tau_k^{(1)}\}$   and $\|\Drhok\|_{\msF}^{} \leq \max\{2\rho,1\} \|\nF_k\|_{\msF}^{}$  that
\be\label{equ:theorem:convergence:a1}
\mF_{k+1} \leq \mF_r^k - \delta \brev{c_2} \|\Drhok\|_{\msF}^2,
\ee
where $\brev{c_2} = \frac{\min\{c,\tau_k^{(1)}\}  \min\{\rho,1\}}{\max\{2\rho,1\}^2}.$
Assume that  $\Drhok = 0$ will not happen after finite iterations. \rev{Numerically, if $\mF_{k+1} < \mF_{\mathrm{best}}$, we must have that
  $\mF_{k+1} \leq \mF_{\mathrm{best}} - \epsilon_{\mathrm{mach}}$, where $\epsilon_{\mathrm{mach}}$ is the machine precision again. By this, we can show \cite{jiang2012framework} that $l^k = L$  holds for infinite number of times.} Define the infinite set $\mathcal{K}\coloneqq\{k_i : l^{k_i} = L\}.$ If the relation \reff{add11} is false, then there must exist a positive constant $\brev{c_3}>0$ such that $\|\Drhok\|_{\msF}^{} \geq \brev{c_3}$ for all sufficiently large $k$. Then we obtain by \reff{equ:theorem:convergence:a1} that
\be \label{equ:theorem:convergence:a2} \mF_{j} \leq \mF_r^{k_i} - \varepsilon_1 \quad\mbox{for all $k_i<j\leq k_{i+1}$,} \ee
where $\varepsilon_1 = \frac{\delta  \brev{c_3} \min\{\rho,1\}\min\{c\cdot \brev{c_3}, \epsilon_{\min}\}}{\max\{2\rho,1\}^2}$ is a positive constant. Moreover, by the definition of $\mF_r$, we know that   $\mF_r^{k_{i+1}} \leq \max_{k_i < j \leq k_{i+1}} \mF_j$,  which with \reff{equ:theorem:convergence:a2} implies that for all $k_i\in \mathcal{K}$,
 $$\mF_{r}^{k_{i+1}} \leq \mF_r^{k_i} - \varepsilon_1. $$
Since $\mathcal{K}$ is an infinite set, the summation of the above relation over $\mathcal{K}$ leads to a contradiction to the boundedness of $\mF(X)$ in the feasible set.  Therefore \reff{add11} must be true. This completes the proof.
$\smartqed$
\end{proof}

Several remarks on the convergence are made here.
\rev{Firstly, we note that} Theorem \ref{theorem:convergence} still holds for \rev{a certain} nonlinear $g(\tau)$, provided that $|g(\tau)/\tau|$ is bounded above.
\rev{Secondly, we point out that inequality \reff{equ:theorem:convergence:well}, the basis for proving \reff{equ:theorem:convergence:lowbound}, is a crucial ingredient in establishing the convergence. Actually, this inequality also holds for the gradient projection, the polar decomposition or the Wen-Yin update scheme, but with different constant before the term $t_k$.} Thus  similar convergence results can be established for the \rev{three} schemes in the framework of Algorithm \ref{algorithm:afbb}.
\rev{Thirdly, \rrev{with the key inequalities in \reff{equ:theorem:convergence:well}}, we can similarly prove the global convergence of the feasible BB-like methods based on the monotone line search, or the nonmonotone line \rrev{searches} proposed in \cite{grippo1986nonmonotone, zhang2004nonmonotone}.}
\rev{Finally,  we know from \reff{equ:algorithm:safeguard} that  the sequence $\{\tau_k: k \geq 0\}$ is bounded above. This, with $\upsilon_k = \tau_k \|\Drhok\|_{\msF}^{}$ and  Theorem \ref{theorem:convergence}, indicates that  either $\upsilon_k =0$ for some finite $k$ or $\liminf\limits_{k \rightarrow \infty} \upsilon_k =0.$ Combining this with  statement (iv) of Lemma \ref{lemma:update},  we have either $\text{cond}(J_k) = 1$ for some finite $k$ or $ \liminf\limits_{k \rightarrow \infty} \text{cond}(J_k) = 1$. }
\subsection{Controlling feasibility errors}\label{subsection:error}
The update scheme \reff{equ:update:new} is so constructed that $Y(\tau)^{\msT}Y(\tau) = I_p$ always holds. In practical computations, however,
the orthogonality constraint may be violated after several iterations. As will be seen, it is mainly due to the numerical errors occured
on the multiplication $X^{\msT}W$, which should be exactly equal to zero in theory. In the following, we give a detailed analysis for this phenomenon and then propose  a strategy for controlling the feasibility errors.

Now we assume that all the arithmetics are exact, but the \brev{orthogonality} constraint \brev{may be} violated at $X$. 
Denote $\Xi_X = X^{\msT}X - I_p$, $\Xi_Y = Y^{\msT}Y - I_p.$ It is easy to verify from the \brev{definition} \brev{\reff{equ:update:new}} of  $J$ that
\be
\rev{(2I_p - J)^{\msT} (2I_p - J) + \tau^2 W^{\msT}W = J^{\msT}J.} \nn
\ee
\rev{Performing respectively the left and right multiplications by $J^{-\msT}$ and $J^{-1}$ on both sides of the above equation, we obtain}
\be \label{equ:error1}
(2J^{-1}-I_p)^{\msT}(2J^{-1}-I_p) + \tau^2 J^{-\msT}W^{\msT}WJ^{-1} = I_p,
\ee
which implies that
\be \|2J^{-1} - I_p\|_2 \leq 1. \label{add12} \ee
Rewrite $Y(\tau)$ in update scheme \reff{equ:update:new} as $Y(\tau) = X(2J^{-1} - I_p) + \tau W J^{-1}$. This,  with  \reff{equ:error1} and $X^{\msT}X = \Xi_X + I_p$, implies  that
\be \label{equ:Xi}
\Xi_Y = (2J^{-1} - I_p)^{\msT} \Xi_X(2J^{-1} - I_p) +\tau(2J^{-1} - I_p)^{\msT}X^{\msT}WJ^{-1} + \tau J^{-T}W^{\msT}X(2J^{-1} - I_p).
\ee
It follows from  $W = -(I_n - XX^{\msT})\Drho$ in \reff{equ:update:new} and the definition of $\Xi_X$ that  $X^{\msT} W  =  \Xi_X X^{\msT}\Drho$. Then by the Cauchy inequality, \reff{equ:Xi},  \reff{add12}, $\|J^{-1}\|_2 \leq 1$ and $\tau \|\Drho\|_{\msF}^{} \leq \epsilon_{\max}$, we know that
\be\label{equ:Xinorm}
\|\Xi_Y\|_{\msF}^{} \leq \left(1 + 2 \tau \|X^{\msT}\Drho\|_{\msF}^{} \right) \|\Xi_X\|_{\msF}^{}  \leq \Big(1 + 2\epsilon_{\max}\rev{\sqrt{1+ \|\Xi_{X}\|_2^{}}} \Big)\|\Xi_X\|_{\msF}^{},
\ee
\brev{showing} that the  feasibility of $Y$ may be out of control after some iterations.

To control the feasibility errors, we propose an approach, in which the matrix $W$ in the scheme is replaced by
\be \label{equ:newW}
 \widehat{W} = -\left(I_n - X(X^{\msT}X)^{-1} X^{\msT}\right)G = -\left(I_n - X(X^{\msT}X)^{-1} X^{\msT}\right)\Drho.
\ee
Notice that when $X^{\msT}X$ is exactly $I_p$, $\widehat{W}$ reduces to $-(I_n - XX^{\msT})\Drho$ which is identical to  the original $W$ in \reff{equ:update:new} corresponding to $E = \Drho$.
Denote by $\widehat{Y}$ the new $Y$ computed by using $\widehat{W}$ and define $\widehat{\Xi}_Y = \rev{\widehat{Y}^{\msT}\widehat{Y}}- I$. Noting that $ X^{\msT} \widehat{W}\equiv 0$, similarly to the deriving of \reff{equ:Xi}, we obtain that
$$\widehat{\Xi}_Y = (2J^{-1} - I_p)^{\msT} \Xi_X(2J^{-1} - I_p ), $$
which with \reff{add12} implies that $\|\widehat{\Xi}_Y\|_{\msF}^{}   \leq \|\Xi_X\|_{\msF}^{}.$
Thus, with $np^2 + 8p^3/3$ extra flops for computing the  $\widehat{W}$ in \reff{equ:newW}, we can efficiently control the cumulative feasibility errors. Considering that $X^{\msT}\widehat{W}$ is not exactly zero in practical computations, we may reasonably assume that
$\|X^{\msT}  - X^{\msT}X (X^{\msT}X)^{-1}X^{\msT}\|_{\msF}^{}$ is in the same order as the machine precision $\varepsilon_{\text{mach}}$. Then we have $\|X^{\msT}\widehat{W}\|_{\msF}^{} \leq O(\varepsilon_{\text{mach}}) \| \Drho\|_{\msF}^{}$. Following the analysis of deriving \reff{equ:Xinorm},  we have that
$$ \|\widehat{\Xi}_Y\|_{\msF}^{} \leq \|\Xi_X\|_{\msF}^{} + 2 \rev{\epsilon_{\max}O(\varepsilon_{\text{mach}})},$$
which significantly improves the bound in \reff{equ:Xinorm}.
 \begin{figure}\label{figure:error}
\centering
\includegraphics[totalheight = 2. in ]{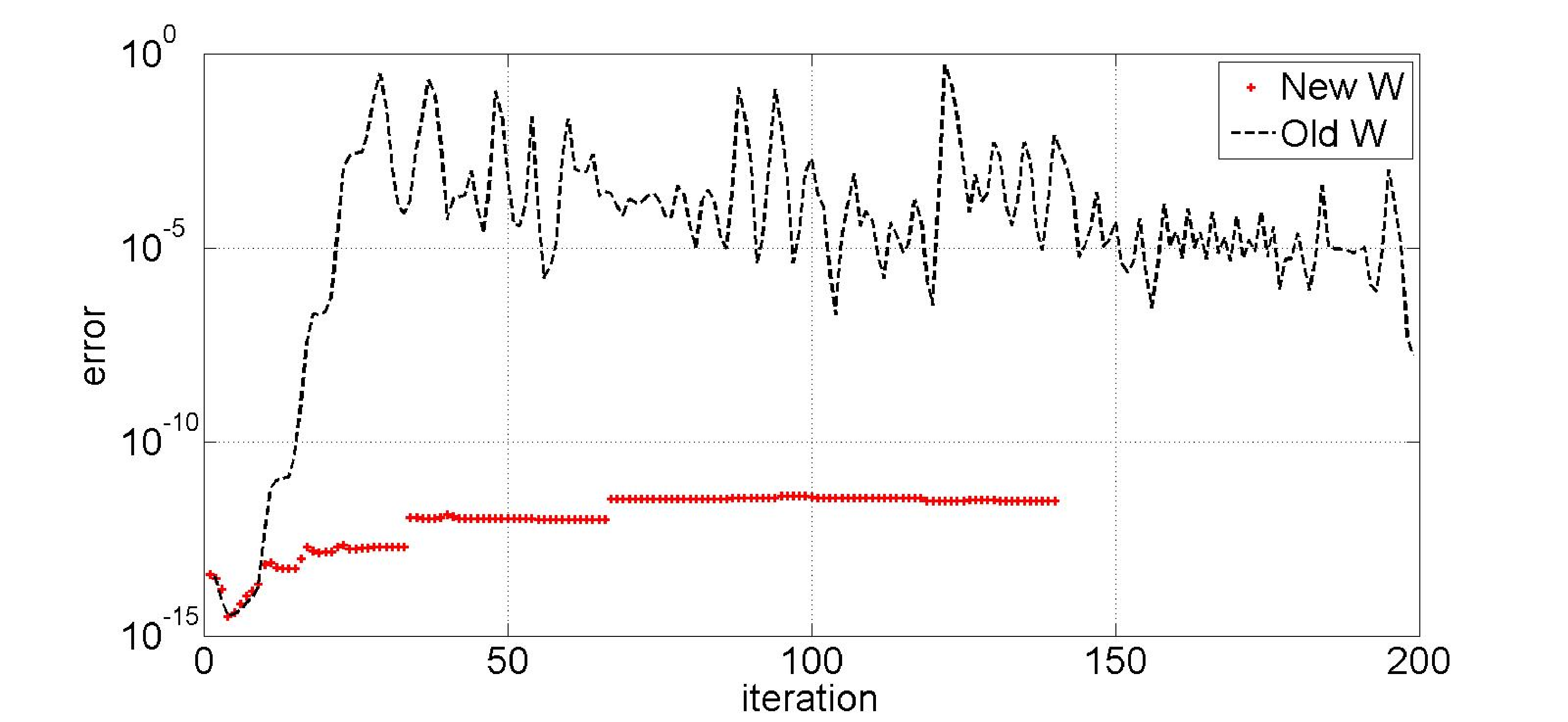}
\caption{The error of feasibility during iterations}
\end{figure}

To illustrate the usefulness of the above strategy, we \brev{take} a typical example, which is to calculate the sum of the four
largest eigenvalue of matrix ``S3DKT3M2'' from UF Sparse Matrix Collection \cite{davis2009university}.  We tried  \brev{two versions of} Algorithm \ref{algorithm:afbb}\brev{, in which the original $W$ and the modified $\widehat{W}$ in \reff{equ:newW} are used respectively.} \brev{The corresponding parameters are chosen according to \S\ref{subsection:stopping}.}
   Figure \ref{figure:error} depicts the error of the feasibility  versus the iteration number. From this figure, we see that the feasibility errors during the iterations are efficiently controlled by using $\widehat{W}$.

\section{Several extensions}\label{section:extension}
In this section, we mention some extensions of the constraint preserving update scheme \reff{equ:update:new} and the AFBB method.
At first, we consider \rev{a special case of \reff{prob:MultiOrthoStiefel}; {\it i.e.}, optimization}
with the \rrev{generalized} orthogonality constraint,
\be \label{prob:GeneralStiefel}
 \min_{X \in \mbC^{n \times p}}  \, \mathcal{F}(X),\ \  \mbox{s.t.} \ \ X^*HX=K,
\ee
where $H \in \mbR^{n \times n}$ is a symmetric positive semidefinite (not necessarily symmetric positive definite) matrix and  $K \in \mbR^{p \times p} $ is a symmetric positive definite matrix. The first-order optimality condition of  problem \reff{prob:GeneralStiefel} is as follows.
\begin{lemma} \label{lemma:generaloptimality}
Suppose $X$ is a local minimizer of problem \reff{prob:GeneralStiefel}. Then $X$ satisfies the first-order optimality condition
$\nabla \mF \coloneqq G - H X G^* X K^{-1} =0$, \rev{with the associated Lagrange multiplier $\Lambda = G^* X K^{-1}$.  Moreover,
$\nabla \mF = 0$ if and only if  $D \coloneqq GX^*H^2X - HXG^*HX = 0$. }
\end{lemma}
\begin{proof}
  \rev{The first statement can be easily verified. For the second one, we first assume that $\nabla \mF = 0$.  Then, we know that  $G = HX K^{-1} X^* G.$ \rrev{Substituting $HXK^{-1} X^*G$ for the second $G$ in} the definition of $D$ yields}
  \be
  \rev{D = (G - HXG^* XK^{-1}) X^*H^2X  = \nabla \mF X^*H^2X = 0. }\nn
  \ee
  \rev{Now we assume that $D = 0$. Noting that $D = (GX^*H - HXG^*) HX$, we can obtain
  \be \label{equ:extension:descent}
 \langle G, D \rangle =  \langle GX^*H, GX^*H - HXG^* \rangle = \frac12 \|GX^*H - HXG^*\|_{\msF}^2,
\ee
which with $D = 0$ means that $GX^*H - HXG^* = 0$. Noting that $\nabla \mF= (GX^*H - HXG^*) XK^{-1}$, we can see that
$\nabla \mF =0$. The proof is completed.
}
 $\smartqed$
\end{proof}

Similar to the analysis in \S\ref{section:update}, we can give the following feasible update scheme for \reff{prob:GeneralStiefel}.
 \begin{lemma} \label{lemma:generalupdate}
   For any feasible point $X$ with $X^*HX=K$, \brev{let} $W = -(I_n - XK^{-1}X^*H)\brev{D}$ and
   $J(\tau)=K + \frac{\tau^2}{4}W^*HW  + \frac{1}{2}\tau X^* H \brev{D}$.  Consider the curve given by
 \be \label{equ:generalupdate2}
  Y(\tau)=(2X + \tau W)J^{-1} K- X.
 \ee
Then we have that
\begin{enumerate}
  \item[\emph{(i)}] $Y(\tau)^* H Y(\tau)= K;$
  \item[\emph{(ii)}] $Y'(0) = -D$ is a descent direction and  $\mF'_{\tau}(Y(0)) \leq -\rev{\frac12 \|XK^{-1}\|_2^{-2}} \|\nabla \mF\|_{\msF}^{\rev{2}};$
  \item[\emph{(iii)}] \rev{$\emph{cond}(J) \leq   \frac{5 +  \upsilon^2}{4} \emph{cond}(K)$, where $\upsilon = \frac{\tau \|H^{\frac12} D\|_{\msF}^{}}{\sqrt{\|K\|_2^{}}}$}.
 \end{enumerate}
\end{lemma}

\begin{proof}
  Since $X^* H \brev{D}$ is skew-Hermitian, $J$ is invertible \rrev{and hence $Y(\tau)$ is well-defined.} \rrev{Rewriting $Y(\tau) = X(2J^{-1}K - I_p) + \tau WJ^{-1}K$ and noting $X^* HW = 0$,  we can easily verify (i).} Moreover, it is not difficulty to verify that $Y'(0) = -D$.
  \rev{It follows from  $\nabla \mF = (GX^*H - HXG^*)XK^{-1}$ that}
\begin{equation}
 \|\nabla \mF\|_{\msF}^{}  \leq \rev{\|XK^{-1}\|_2^{}}\|GX^*H - HXG^*\|_{\msF}^{}.
\nn
\end{equation}
\rev{Combing the above inequality with  \reff{equ:extension:descent},  $\mF_{\tau}'(Y(0)) =\langle G, Y'(0) \rangle$ and $Y'(0)= - D$,} we have that
$$\mF'_{\tau}(Y(0)) =  \rrev{-\frac12 \|GX^*H - HXG^*\|_{\msF}^2} \leq - \rev{\frac12 \|XK^{-1}\|_2^{-2}} \|\nabla \mF\|_{\msF}^{2}.$$
So (ii) is true.

For (iii), \rev{it follows from $D = -W + XK^{-1} X^* HD$ and $X^*HW = 0$ that $D^*HD = W^*HW + D^*HXK^{-1}X^*HD$. Thus there holds}
\be
\rev{ \|H^{\frac12} D\|_{\msF}^2 = \|H^{\frac12} W\|_{\msF}^2 + \|K^{-\frac12} X^*HD\|_{\msF}^2.} \nn
\ee
\rrev{Meanwhile, it is easy to see that
  $$\|J\|_2  \leq  \|K\|_2^{}  + \frac{\tau^2}{4}\|H^{\frac12}W\|_{\msF}^2 + \frac{\tau}{2}\|K^{\frac12}\|_{2}\|K^{-\frac12} X^*HD\|_{\msF}^{}.$$}
  \rrev{Plugging $\tau = \frac{\sqrt{\|K\|_2}}{\|H^{\frac12}D\|_{\msF}}v$ into the above equality yields}
\begin{align}\label{equ:generalupdate:a1}
  \rev{\|J\|_2} &  \rev{\leq \|K\|_2 + \frac{\|K\|_2 \upsilon^2}{4} \cdot \frac{\|H^{\frac12}W\|_{\msF}^2}{\|H^{\frac12}D\|_{\msF}^2} + \frac{\|K\|_2\upsilon}{2} \cdot \frac{\|K^{-\frac12} X^*HD\|_{\msF}}{\|H^{\frac12}D\|_{\msF}}} \nn \\
  & \rev{\leq \|K\|_2 + \|K\|_2\cdot \max_{0 \leq t \leq 1} \left( \frac{\upsilon^2}{4} ( 1 - t^2) +  \frac{\upsilon}{2} t \right)} \nn \\
  & \rev{\leq \frac{5 + \upsilon^2}{4} \|K\|_2.}
\end{align}
\rev{On the other hand, we see from \reff{equ:generalupdate2}, $X^*HW=0$ and $X^*HX = K$ that  $2J^{-1} = K^{-1}(X^*HY)K^{-1} + K^{-1}$, which indicates that $\|J^{-1}\|_2 \leq \|K^{-1}\|_2$. Combing this with \reff{equ:generalupdate:a1} gives (iii). }
The proof is completed.
$\smartqed$
\end{proof}

\rev{ }With the help of the above lemma, we can design the AFBB method for problem \reff{prob:GeneralStiefel} similar to Algorithm \ref{algorithm:afbb}.  Global convergence results can also be established.

\rev{Finally, we point out that our scheme and algorithm can naturally be extended to \reff{prob:MultiOrthoStiefel},} since the variables $X_1, \ldots, X_p$ are separated in the constraints.

\section{Numerical Results}\label{section:numericaltests}
In this section, we present numerical results  on a variety of problems to illustrate the  efficiency of our AFBB method. We implemented AFBB in Matlab R2012a. All  experiments were performed  in  Matlab under a Linux operating system  on a  Thinkpad T420 Laptop with an Intel$^\circledR$ dual core CPU at 2.60GHz $\times$ 2 and  4GB of  RAM.

\subsection{Stopping criteria}\label{subsection:stopping}
Define $\text{tol}_k^x\coloneqq  \frac{\|X_k - X_{k-1}\|_{\msF}^{}}{\sqrt{n}}$ and $\text{tol}_k^f: = \frac{|\mF_k - \mF_{k+1}|}{|\mF_k| + 1}$. We  terminate  the algorithm if one of the following holds: (i) $k\geq \text{MaxIter}$; (ii) $\|\Drhok\|_{\msF}^{} \leq \epsilon \|D_{\rho,0}\|_{\msF}^{}$; (iii) $\text{tol}_k^x  \leq \epsilon_x$ and $\text{tol}_k^f  \leq \epsilon_f$; (iv)  $\mathrm{mean}([\text{tol}_{k - \min\!\mathrm{\{}k,T\mathrm{\}} + 1}^x,$ $\ldots, \text{tol}_k^x ]) \leq 10 \epsilon_x$ and  $\text{mean}([\text{tol}_{k - \min\!\mathrm{\{}k,T\mathrm{\}} + 1}^x, \ldots, \text{tol}_k^x ]) \leq 10 \epsilon_f,$ where $\text{MaxIter}$ is the maximal iteration number. These criteria are the same as those used in \cite{wen2013feasible} except that we replace $\|\nabla \mF_k\|_{\msF}^{} \leq \epsilon$ \brev{therein}  by criterion (ii) here.
Unless otherwise specified, we set $\epsilon = 10^{-5}$, $\epsilon_x = 10^{-5}$, $\epsilon_f = 10^{-8}$, $T = 5$ \brev{and}  $\text{MaxIter} = 3000$.

The other parameters for AFBB  are given as follows:
$$ \rho = \rev{0.25}, \  \sigma = 0.5, \ \delta = 0.001, \ \epsilon_{\min}= \rev{10^{-8}}, \ \epsilon_{\max}= \rev{10^{8}}, \ \Delta = \rev{10^{10}}, \ L = 3.$$
The initial trial stepsize $\tau_0^{(1)}$ at the first iteration is chosen to be  $0.5\,\|D_{\rho,0}\|_{\msF}^{-1}$.  The ABB stepsize  \reff{equ:Stepsize:ABB}  is used for computing  the trial stepsize  $\tau_k^{(0)}$. The ``tic-toc'' command in Matlab is used
to obtain the CPU time in seconds elapsed by each code. \rev{After the iterations, if $\|X^{\msT} X - I_p\|_{\msF}^{} \geq 10^{-14}$, a re-orthogonality procedure will be performed to enhance the feasibility.}

\subsection{Nearest low-rank correlation matrix problem}
Given  $C \in \mathcal{S}^n$  and a nonnegative weight matrix $H \in \mathcal{S}^n$, the nearest low-rank correlation matrix problem is given by
\be \label{prob:NLCM}
\min_{X \in \mathcal{S}^n}\,   \frac{1}{2}\|H\odot (X-C)\|_{\msF}^2, \  \ \mbox{s.t.} \  \  \mbox{diag}(X) = e,\  \ \mbox{rank}(X) \leq r,\ \ X \succeq 0,
\ee
where $\odot$ is the Hadama product operator of two matrices, $e\in \mbR^n$ is the vector with all ones,  $r<n$ is a given positive integer number. A usual weight is  $H=\textbf{1}$, where $\textbf{1} \in \mbR^{n \times n}$  represents the matrix with all ones.

  To deal with the nonconvex rank constraints $\mbox{rank}(X) \leq r$, as used in the geometric optimization method  \cite{grubisic2007efficient}, majorization method \cite{pietersz2004rank}, and trigonometric parametrization method  \cite{Rebonato1999most},  we rewrite $X = V^{\msT}V $ with $V = [V_1, \ldots, V_n] \in \mbR^{r \times n}$. Consequently, we get the equivalent formulation of \reff{prob:NLCM} as follows:
 \be \label{prob:NLCM2}
 \min_{V \in \mbR^{r \times n}}  \theta(V; H, C)\coloneqq \frac{1}{2}\|H\odot (V^{\msT}V-C)\|_{\msF}^2, \ \ \mbox{s.t.}\ \  \|V_i\|_2 = 1,\ \  i = 1, \ldots, n,
\ee
which is the  minimization of a quartic \brev{polynomial} over spheres. Among many  approaches to solve problem \reff{prob:NLCM},  we  compared our algorithm with several state-of-the-art methods; \emph{i.e.}, the majorized penalty approach  (PenCorr\footnote{It can be downloaded from \url{http://www.math.nus.edu.sg/~matsundf/#Codes}.}) \cite{gao2010majorized}, the sequential semismooth Newton  method (SemiNewton) \cite{li2011sequential},  and the  Wen-Yin nonmonotone BB method (OptM\footnote{It can be downloaded from \url{http://optman.blogs.rice.edu/}.}) proposed in \cite{wen2013feasible}. For more comprehensive literature reviews,  see  \cite{gao2010majorized,li2011sequential}.

We chose the initial point of problem \reff{prob:NLCM2} in the same way as in the majorization method. Specifically, we selected the  modified PCA \cite{flury1988common} of $C$; {\it i.e.,} $C_{\brev{\mathrm{pca}}}$, to be  the initial point in AFBB or OptM. Let $C$ have the eigenvalue decomposition $C = P\Diag(\lambda_1, \ldots, \lambda_n) P^{\msT},$ where $P^{\msT}P = I_n$ and $\lambda_1 \geq \cdots \geq \lambda_n.$ Define $\Lambda_r = \text{Diag}(\lambda_1, \ldots, \lambda_r)$ and denote by  $P_1$ the first $r$ columns of $P$. Then the $i$-th column of $C_{\brev{\mathrm{pca}}}$ is $[C_{\brev{\mathrm{pca}}}]_{(i)} =z_i/\|z_i\|$,  where $z_i = P_1 [\Lambda_r^{1/2}]_{(i)},\, i = 1, \ldots, r. $
In case of $\lambda_i \leq 0$ for some $1\leq i \leq r$, we first solved the following problem:
\begin{displaymath}
    \widetilde{C} = \argmin\limits_{X \in \mathcal{S}^n} \  \frac{1}{2}\|X-C\|_{\msF}^2, \  \ \mbox{s.t.}  \ \  \mbox{diag}(X) = e,\ \  X \succ 0
  \end{displaymath}
  \rev{by the  semismooth Newton method \cite{qi2007quadratically} and then} set the initial point as $\widetilde{C}_{\brev{\mathrm{pca}}}.$ \rev{Note that SemiNewton and PenCorr also have their own efficient ways to generate good initial points.}

We list the test problems  from \cite{gao2010majorized, jiang2012inexact, li2010inexact, li2011sequential, pietersz2004rank}  as follows.
\rev{For the case when $H \ne \textbf{1}$, we may need to use the incomplete $C$. More specifically, the entries of $C$, corresponding to the nonzero weights, will be set to be zeros.}

{Ex. 1}\upcite{li2011sequential}: The matrix $C$ is the $387 \times 387$ one-day correlation matrix (as of Oct. 10, 2008) from the lagged datasets of RiskMetrics.\footnote{Dr. Qingna Li provided us this matrix kindly.}

{Ex. 2}\upcite{pietersz2004rank}:  $H = \textbf{1}, n = 500$, the entries
$$C_{ij} = \exp\left(-\gamma_1 |i-j| - \frac{\gamma_2 |t_i - t_j|}{\max\{i,j\}^{\gamma_3}  - \gamma_4 |\sqrt{i} - \sqrt{j}\,|}\right)$$
for $i,j = 1, \ldots, n$ with $\gamma_1 = 0, \gamma_2 = 0.480, \gamma_3 = 1.511, \gamma_4 = 0.186.$

{Ex. 3}\upcite{gao2010majorized}: $n = 500$, the entries $C_{ij} = 0.5 + 0.5 e^{-0.05|i-j|}$ for $i,j = 1, \ldots, n$. The weight matrix $H$ is either $\textbf{1}$ or a random matrix whose entries are uniformly distributed in $[0.1, 10]$ except for $200$ entries in $[0.01, 100]$.

{Ex. 4}\upcite{gao2010majorized}: $n = 943, C$ is based on $100,000$ ratings for $1682$ movies by $943$ users from Movielens data sets. It can be download from \url{http://www.grouplens.org/node/73}. The weight matrix $H$ is either $\textbf{1}$ or the one provided by T. Fushiki at Institute   of  Statistical Mathematics in Japan.

{Ex. 5}\,-\,9\upcite{jiang2012inexact, li2010inexact}: We consider the five gene correlation matrices $\widehat{C}$: Lymph, ER, Arabidopsis, Leukemia and Hereditary bc. For the sake of comparison, as done in \cite{jiang2012inexact}, we perturb $\widehat{C}$  to $$C = (1 - \gamma) \widehat{C} + \gamma F,$$ where $\gamma = 0.05$ and $F$ is a random symmetric matrix with entries  uniformly distributed in $[-1, 1]$.
The corresponding weight matrix $H$ is either $\textbf{1}$ or the one created by  Example 2 in \cite{jiang2012inexact}.

\begin{table}[!htbp]
\linespread{1.5}
  \centering
  \caption{\rev{Numerical results of  the nearest low-rank correlation matrix problem  with $H = \textbf{1}$: SemiNewton, PenCorr}} \label{table:NLCMpart1}
~\\
  \begin{scriptsize}
    \revmini{\begin{tabular}{@{}l@{\hspace{4.3mm}}r@{\hspace{4.3mm}}r@{\hspace{4.3mm}}r@{\hspace{4.3mm}}r@{\hspace{0.3mm}}r@{\hspace{5mm}}r@{\hspace{4.3mm}}r@{\hspace{4.3mm}}r@{\hspace{4.3mm}}r@{}}
      \Xhline{0.6pt}
      \Gape[8pt]      &  \multicolumn{4}{c}{SemiNewton} &&  \multicolumn{4}{c}{PenCorr} \\
      \Xcline{2-5}{0.6pt} \Xcline{7-10}{0.6pt}
  \Gape[8pt] r& nlcmres$^0$ & nlcmres\,  &  time & feasi\, && nlcmres$^0$ &   nlcmres\,  &  time & feasi\,    \\
  \Xhline{1.1pt}
\multicolumn{10}{@{} l @{} }{       {Ex. 1},  $\ n =  387$} \\
  2 & 1.653e02 & 1.630e02  &  3.6 & 4e-15 &&  1.653e02 &  1.623e02 & 12.3 & 1e-08  \\
  5 & 7.026e01 & 6.157e01  &  2.2 & 4e-15 &&  7.026e01 &  6.111e01 &  8.1 & 5e-08  \\
 20 & 8.641e00 & 6.087e00  &  1.5 & 5e-15 &&  8.641e00 &  6.066e00 &  2.3 & 6e-07  \\
 40 & 1.335e00 & 7.765e-01  &  1.3 & 7e-15 &&  1.335e00 &  7.768e-01 &  1.1 & 2e-08  \\
 80 & 5.889e-02 & 3.228e-02  &  1.2 & 9e-15 &&  5.889e-02 &  3.262e-02 &  0.7 & 1e-08  \\
\multicolumn{10}{@{} l @{} }{       {Ex. 4},  $\ n =  943$} \\
  5 & 4.410e02 & 4.136e02  & 18.8 & 6e-15 &&  4.410e02 &  4.128e02 & 73.6 & 6e-08  \\
 20 & 2.943e02 & 2.888e02  & 22.7 & 7e-15 &&  2.943e02 &  2.887e02 & 36.4 & 3e-08  \\
 50 & 2.765e02 & 2.773e02  & 26.6 & 1e-14 &&  2.765e02 &  2.763e02 & 18.7 & 8e-08  \\
 100 & 2.758e02 & 2.761e02  & 26.6 & 1e-14 &&  2.758e02 &  2.758e02 &  4.6 & 2e-08  \\
 200 & 2.758e02 & 2.758e02  & 20.1 & 1e-14 &&  2.758e02 &  2.758e02 &  4.5 & 2e-08  \\
 250 & 2.758e02 & 2.758e02  & 27.3 & 1e-14 &&  2.758e02 &  2.758e02 &  4.5 & 2e-08  \\
\multicolumn{10}{@{} l @{} }{ {Ex. 8}, Leukemia,  $\ n = 1255$} \\
  5 & 3.918e02 & 3.317e02  & 45.4 & 7e-15 &&  3.918e02 &  3.309e02 & 140.6 & 2e-08  \\
 20 & 1.554e02 & 1.055e02  & 32.6 & 8e-15 &&  1.554e02 &  1.055e02 & 87.0 & 2e-07  \\
 50 & 6.529e01 & 4.473e01  & 29.2 & 1e-14 &&  6.529e01 &  4.473e01 & 33.4 & 4e-07  \\
 100 & 3.937e01 & 3.274e01  & 21.3 & 2e-14 &&  3.937e01 &  3.274e01 & 28.9 & 4e-08  \\
 200 & 3.163e01 & 3.099e01  & 35.7 & 2e-14 &&  3.163e01 &  3.095e01 & 12.3 & 7e-08  \\
 400 & 3.078e01 & 3.078e01  & 29.3 & 3e-14 &&  3.078e01 &  3.078e01 &  7.6 & 2e-07  \\
\multicolumn{10}{@{} l @{} }{{Ex. 9}, Hereditary bc,  $\ n = 1869$} \\
  5 & 4.657e02 & 4.361e02  & 118.7 & 8e-15 &&  4.657e02 &  4.357e02 & 342.8 & 6e-08  \\
 20 & 7.194e01 & 6.426e01  & 104.2 & 1e-14 &&  7.194e01 &  6.425e01 & 163.0 & 9e-08  \\
 50 & 5.823e01 & 5.200e01  & 79.1 & 2e-14 &&  5.823e01 &  5.199e01 & 112.1 & 8e-08  \\
 100 & 5.311e01 & 5.026e01  & 92.5 & 2e-14 &&  5.311e01 &  5.026e01 & 109.2 & 3e-08  \\
 200 & 4.990e01 & 4.989e01  & 110.9 & 3e-14 &&  4.990e01 &  4.970e01 & 36.5 & 6e-07  \\
 400 & 4.966e01 & 4.966e01  & 27.6 & 1e-09 &&  4.966e01 &  4.966e01 & 23.7 & 4e-09  \\
 \Xhline{0.6pt}
  \end{tabular}}
\end{scriptsize}
\end{table}
\begin{table}[!htbp]
\linespread{1.6}
  \centering
  \caption{\rev{Numerical results of  the nearest low-rank correlation matrix problem  with $H = \textbf{1}$:  OptM, AFBB, AFBB-rand$X_0$}} \label{table:NLCMpart1_2}
~\\
 \begin{scriptsize}
   \revmini{ \begin{tabular}{ @{}l@{\hspace{1.2mm}}r@{\hspace{1.2mm}}r@{\hspace{1.2mm}}r@{\hspace{1.2mm}}r@{\hspace{1.2mm}}r@{\hspace{.3mm}}r@{\hspace{3mm}}r@{\hspace{1.2mm}}r@{\hspace{1.2mm}}r@{\hspace{1.2mm}}r@{\hspace{1.2mm}}r@{\hspace{0.3mm}}r@{\hspace{3mm}}r@{\hspace{1.2mm}}r@{\hspace{1.2mm}}r@{\hspace{1.2mm}}r@{\hspace{1.2mm}}r@{}}
      \Xhline{0.6pt}
      \Gape[8pt]   &\multicolumn{5}{c}{OptM} &&\multicolumn{5}{c}{AFBB}&&\multicolumn{5}{c}{AFBB-rand$X_0$}\\
      \Xcline{2-6}{0.6pt} \Xcline{8-12}{0.6pt} \Xcline{14-18}{0.6pt}
   \Gape[8pt]  r & nlcmres$^0$ & nlcmresi\,  &  time & feasi\,  & nfge &&nlcmres$^0$ & nlcmres\,  & time & feasi\,  & nfge &&nlcmres$^0$ & nlcmres\, & time & feasi\,  & nfge\\
   \Xhline{1.1pt}
\multicolumn{18}{@{} l @{} }{       {Ex. 1},  $\ n =  387$} \\
   2 & 1.653e02 & 1.624e02  &  0.1 & 3e-15 &    31  &&  1.653e02 &  1.624e02 &  0.1 & 3e-15 &    32  &&  3.265e02 &  1.636e02 &  0.3 & 3e-15 &    88  \\
   5 & 7.026e01 & 6.109e01  &  0.1 & 3e-15 &    90  &&  7.026e01 &  6.109e01 &  0.1 & 3e-15 &    80  &&  2.995e02 &  6.109e01 &  0.2 & 3e-15 &   114  \\
  20 & 8.641e00 & 6.063e00  &  0.2 & 3e-15 &   145  &&  8.641e00 &  6.063e00 &  0.2 & 3e-15 &   125  &&  2.902e02 &  6.063e00 &  0.4 & 4e-15 &   273  \\
  40 & 1.335e00 & 7.756e-01  &  0.4 & 4e-15 &   166  &&  1.335e00 &  7.759e-01 &  0.3 & 3e-15 &   151  &&  2.896e02 &  7.755e-01 &  0.9 & 5e-15 &   478  \\
  80 & 5.887e-02 & 3.302e-02  &  0.3 & 3e-15 &    79  &&  5.887e-02 &  3.266e-02 &  0.3 & 3e-15 &    85  &&  2.895e02 &  4.705e-02 &  2.6 & 6e-15 &   876  \\
\multicolumn{18}{@{} l @{} }{       {Ex. 4},  $\ n =  943$} \\
   5 & 4.359e02 & 4.128e02  &  1.0 & 5e-15 &   120  &&  4.359e02 &  4.128e02 &  0.8 & 5e-15 &   100  &&  7.282e02 &  4.128e02 &  1.0 & 5e-15 &   147  \\
  20 & 2.930e02 & 2.887e02  &  2.3 & 6e-15 &   340  &&  2.930e02 &  2.887e02 &  2.2 & 6e-15 &   321  &&  7.083e02 &  2.887e02 &  1.8 & 7e-15 &   244  \\
  50 & 2.843e02 & 2.763e02  &  3.0 & 7e-15 &   252  &&  2.843e02 &  2.763e02 &  2.4 & 7e-15 &   209  &&  7.031e02 &  2.763e02 &  2.7 & 9e-15 &   243  \\
  100 & 2.948e02 & 2.758e02  &  3.3 & 7e-15 &   125  &&  2.948e02 &  2.758e02 &  2.3 & 7e-15 &   105  &&  7.023e02 &  2.758e02 &  2.8 & 1e-14 &   131  \\
  200 & 3.088e02 & 2.758e02  &  4.4 & 7e-15 &   109  &&  3.088e02 &  2.758e02 &  3.7 & 7e-15 &    93  &&  7.013e02 &  2.758e02 &  4.7 & 2e-14 &   109  \\
  250 & 3.129e02 & 2.758e02  &  5.8 & 7e-15 &   109  &&  3.129e02 &  2.758e02 &  4.4 & 7e-15 &    83  &&  7.012e02 &  2.758e02 &  5.6 & 2e-14 &   113  \\
\multicolumn{18}{@{} l @{} }{ {Ex. 8}, Leukemia,  $\ n = 1255$} \\
   5 & 3.906e02 & 3.309e02  &  1.6 & 6e-15 &   101  &&  3.906e02 &  3.309e02 &  1.6 & 6e-15 &    97  &&  9.105e02 &  3.309e02 &  1.7 & 6e-15 &   116  \\
  20 & 1.506e02 & 1.055e02  &  2.4 & 7e-15 &   164  &&  1.506e02 &  1.055e02 &  2.1 & 7e-15 &   137  &&  8.789e02 &  1.055e02 &  4.1 & 7e-15 &   360  \\
  50 & 5.703e01 & 4.473e01  &  2.1 & 7e-15 &    91  &&  5.703e01 &  4.473e01 &  1.8 & 7e-15 &    69  &&  8.740e02 &  4.473e01 &  5.9 & 1e-14 &   368  \\
  100 & 3.291e01 & 3.273e01  &  3.9 & 7e-15 &   118  &&  3.291e01 &  3.273e01 &  4.1 & 7e-15 &   119  &&  8.712e02 &  3.274e01 & 14.5 & 1e-14 &   557  \\
  200 & 4.583e01 & 3.095e01  &  8.0 & 7e-15 &   143  &&  4.583e01 &  3.095e01 &  5.7 & 7e-15 &    99  &&  8.706e02 &  3.095e01 & 23.8 & 2e-14 &   455  \\
  400 & 7.580e01 & 3.078e01  & 21.7 & 8e-15 &   188  &&  7.580e01 &  3.078e01 & 22.9 & 7e-15 &   207  &&  8.704e02 &  3.078e01 & 30.2 & 2e-14 &   266  \\
\multicolumn{18}{@{} l @{} }{{Ex. 9},  Hereditary bc,  $\ n = 1869$} \\
   5 & 4.651e02 & 4.357e02  &  3.5 & 7e-15 &    56  &&  4.651e02 &  4.357e02 &  3.7 & 7e-15 &    63  &&  1.362e03 &  4.358e02 &  4.5 & 7e-15 &   101  \\
  20 & 6.936e01 & 6.425e01  &  3.9 & 8e-15 &    64  &&  6.936e01 &  6.425e01 &  3.9 & 8e-15 &    65  &&  1.317e03 &  6.425e01 &  5.6 & 9e-15 &   141  \\
  50 & 5.256e01 & 5.199e01  &  7.0 & 8e-15 &   138  &&  5.256e01 &  5.199e01 &  7.5 & 8e-15 &   148  &&  1.309e03 &  5.200e01 & 16.5 & 1e-14 &   440  \\
  100 & 5.506e01 & 5.026e01  &  8.3 & 8e-15 &   113  &&  5.506e01 &  5.026e01 &  7.7 & 8e-15 &   103  &&  1.305e03 &  5.027e01 & 26.0 & 1e-14 &   473  \\
  200 & 8.384e01 & 4.970e01  & 13.7 & 8e-15 &   114  &&  8.384e01 &  4.970e01 & 11.8 & 8e-15 &    99  &&  1.304e03 &  4.970e01 & 31.9 & 2e-14 &   321  \\
  400 & 1.345e02 & 4.966e01  & 50.4 & 8e-15 &   250  &&  1.345e02 &  4.966e01 & 42.4 & 7e-15 &   211  &&  1.304e03 &  4.966e01 & 57.2 & 2e-14 &   288  \\
 \Xhline{0.6pt}
  \end{tabular}}
\end{scriptsize}
\end{table}

In our implementation, all the parameters of each solver  were set to be their default values except the termination criteria of OptM were changed  to  the ones in \S\ref{subsection:stopping}. The numerical results are reported in Tables \ref{table:NLCMpart1} -- \ref{table:NLCMpart2}. \revmini{To save space, we only report some representative results. For more details, one can see \cite{Daijiang2012}.} In \brev{these} tables, \rev{``nlcmres$^0$'' and ``nlcmres'' represent the initial and returned residual $\|H\odot (X-C)\|_{\msF}^{}$, respectively.} The terms  ``time'' and ``feasi''  denote the CPU time  and the violation of $\mbox{diag}(X) = e$, respectively, and  ``nfge''  stands for the total number of function and gradient evaluations. Note that CPU time of  OptM or AFBB includes the CPU time for generating the initial point by the modified PCA.  \rev{To further show the efficiency of our approach, we considered  to implement  AFBB starting from a random $X_0$, denoted by ``AFBB-rand$X_0$''. \revmini{The results of AFBB-rand$X_0$ for $H = \textbf{1}$ are reported} in the last five columns of Table  \ref{table:NLCMpart1_2}. }

From the results,  we know that  AFBB performs better than OptM in terms of the residual, the CPU time and the number of function and gradient evaluations. Again,  in the case when $H = \textbf{1}$,  AFBB not only  runs considerably faster than  SemiNewton and PenCorr,  but also always  find a better solution in terms of the residual except for the problem with large $r$; in the case when $H \ne \textbf{1}$, our AFBB shows great advantage over PenCorr in terms of the solution quality and CPU time.  \rev{Besides, although AFBB-rand$X_0$ performs worse than AFBB, it still can solve the problem in a reasonable time and is comparable with SemiNewton and PenCorr.} Apart from the ABB stepsize, another reason for the efficiency of AFBB is its low complexity cost per iteration. The dominated cost of AFBB at each iteration is computing  the function value and the gradient in \reff{prob:NLCM}. For special $H = \textbf{1}$ and general $H$, the costs  are  $2n^2r + 3nr^2$ and $3n^2r$, respectively. In contrast,  SemiNewton and PenCorr need to solve iteratively a series of least squares correlation matrix problems without the rank constraint, whose cost is very expensive.

\begin{table}[!htbp]
\linespread{1.6}
  \centering
  \caption{Numerical results of the nearest low-rank correlation matrix problem with $H \ne \textbf{1}$: PenCorr, OptM,  AFBB} \label{table:NLCMpart2}
  ~\\
  \begin{scriptsize}
    \revmini{\begin{tabular}{ @{}l@{\hspace{1.2mm}}r@{\hspace{1.2mm}}r@{\hspace{1.2mm}}r@{\hspace{1.2mm}}r@{\hspace{0.3mm}}r@{\hspace{3.5mm}}r@{\hspace{1.2mm}}r@{\hspace{1.2mm}}r@{\hspace{1.2mm}}r@{\hspace{1.2mm}}r@{\hspace{0.3mm}}r@{\hspace{3.5mm}}r@{\hspace{1.2mm}}r@{\hspace{1.2mm}}r@{\hspace{1.2mm}}r@{\hspace{1.2mm}}r@{} }
\Xhline{0.6pt}
\Gape[8pt]   &  \multicolumn{4}{c}{PenCorr} &&  \multicolumn{5}{c}{OptM} &&  \multicolumn{5}{c}{AFBB}\\
\Xcline{2-5}{0.6pt} \Xcline{7-11}{0.6pt} \Xcline{13-17}{0.6pt}
\Gape[8pt]   r & nlcmres$^0$ & nlcmres\,&  time & feasi\, && nlcmres$^0$ &   nlcmres\, &  time & feasi\,  & nfge  && nlcmres$^0$ &  nlcmres\,  &  time & feasi\,  & nfge\\
\Xhline{1.1pt}
   \multicolumn{17}{@{} l @{} }{       {Ex. 3}, $\ n =  500$, random  $H$} \\
  2 &  1.198e03  & 9.122e02 & 89.6 &  3e-09 &&   1.198e03  & 9.120e02 &  0.4 &  3e-15&    50 &&  1.198e03  & 9.120e02 & 0.5 &  4e-15 &   48  \\
  5 &  7.944e02  & 4.541e02 & 53.3 &  8e-09 &&   7.944e02  & 4.539e02 &  1.5 &  3e-15&   201 &&  7.944e02  & 4.539e02 & 0.7 &  3e-15 &   85  \\
 20 &  2.293e02  & 8.812e01 & 38.8 &  3e-07 &&   2.293e02  & 8.809e01 &  1.6 &  4e-15&   185 &&  2.293e02  & 8.809e01 & 1.3 &  4e-15 &  156  \\
 50 &  8.607e01  & 2.195e01 & 48.4 &  9e-07 &&   8.607e01  & 2.189e01 &  7.7 &  4e-15&   787 &&  8.607e01  & 2.189e01 & 7.1 &  4e-15 &  731  \\
 100 &  4.084e01  & 6.445e00 & 55.5 &  3e-07 &&   4.084e01  & 6.295e00 &  38.5 &  4e-15&  3116 &&  4.084e01  & 6.289e00 & 38.1 &  4e-15 & 3011  \\
 125 &  3.184e01  & 4.057e00 & 73.9 &  9e-07 &&   3.184e01  & 3.958e00 &  43.8 &  4e-15&  2948 &&  3.184e01  & 3.939e00 & 42.7 &  4e-15 & 3008  \\
\multicolumn{17}{@{} l@{} }{       {Ex. 4}, $\ n =  943$, $H$ is given by T.  Fushiki} \\
  5 &  1.590e04  & 1.148e04 & 1188.4 &  6e-07 &&   1.541e04  & 1.139e04 &  31.2 &  5e-15&  1525 &&  1.541e04  & 1.139e04 & 39.7 &  5e-15 & 1951  \\
 20 &  7.722e03  & 5.219e03 & 1072.5 &  3e-07 &&   8.005e03  & 5.194e03 &  17.4 &  6e-15&   748 &&  8.005e03  & 5.195e03 & 18.8 &  6e-15 &  821  \\
 50 &  5.211e03  & 3.719e03 & 706.5 &  2e-09 &&   6.126e03  & 3.712e03 &  14.6 &  7e-15&   533 &&  6.126e03  & 3.712e03 & 16.3 &  7e-15 &  539  \\
 100 &  4.844e03  & 3.507e03 & 412.0 &  3e-07 &&   6.443e03  & 3.503e03 &  14.4 &  7e-15&   338 &&  6.443e03  & 3.503e03 & 11.6 &  7e-15 &  269  \\
 200 &  4.844e03  & 3.505e03 & 378.5 &  3e-07 &&   7.621e03  & 3.501e03 &  16.9 &  8e-15&   318 &&  7.621e03  & 3.501e03 & 13.8 &  8e-15 &  275  \\
 250 &  4.844e03  & 3.505e03 & 372.2 &  3e-07 &&   8.023e03  & 3.501e03 &  16.7 &  8e-15&   273 &&  8.023e03  & 3.501e03 & 18.5 &  8e-15 &  324  \\
\multicolumn{17}{@{} l @{} }{{Ex. 8},   Leukemia, $\ n = 1255$, $H$ is generated by Example 2 in \cite{jiang2012inexact}} \\
  5 &  8.545e04  & 7.148e04  &44049.6 &5e-07 &&   8.517e04  & 7.133e04 &  71.3 &  6e-15&  1572 &&  8.517e04  & 7.133e04 & 16.5 &  6e-15 &  352  \\
 20 &  3.320e04  & 2.203e04 & 22685.8 &  1e-07 &&   3.218e04  & 2.200e04 &  52.4 &  7e-15&  1026 &&  3.218e04  & 2.200e04 & 12.6 &  6e-15 &  221  \\
 50 &  1.384e04  & 8.846e03 & 14046.2 &  8e-08 &&   1.225e04  & 8.838e03 &  42.1 &  9e-15&   703 &&  1.225e04  & 8.835e03 & 19.7 &  7e-15 &  332  \\
 100 &  7.615e03  & 6.245e03 & 6607.0 &  1e-07 &&   7.182e03  & 6.243e03 &  32.1 &  7e-15&   470 &&  7.182e03  & 6.242e03 & 30.4 &  7e-15 &  451  \\
 200 &  6.056e03  & 5.914e03 & 517.6 &  7e-09 &&   9.945e03  & 5.917e03 &  102.5 &  1e-14&  1089 &&  9.945e03  & 5.915e03 & 56.8 &  8e-15 &  611  \\
 400 &  6.045e03  & 5.912e03 & 159.9 &  2e-07 &&   1.637e04  & 5.916e03 &  165.1 &  1e-14&  1109 &&  1.637e04  & 5.913e03 & 77.4 &  8e-15 &  529  \\
\multicolumn{17}{@{} l @{} }{{Ex. 9}, Hereditary bc, $\ n = 1869$, $H$ is generated by Example 2 in \cite{jiang2012inexact}} \\
  5 &  1.007e05  & 9.399e04 & 89093.5 &  2e-07 &&   1.006e05  & 9.384e04 &  192.9 &  7e-15&  1725 &&  1.006e05  & 9.384e04 & 33.3 &  7e-15 &  276  \\
 20 &  1.580e04  & 1.381e04 & 51569.0 &  9e-07 &&   1.518e04  & 1.380e04 &  77.4 &  8e-15&   622 &&  1.518e04  & 1.379e04 & 20.8 &  8e-15 &  145  \\
 50 &  1.197e04  & 1.069e04 & 55899.2 &  1e-07 &&   1.144e04  & 1.068e04 &  57.1 &  8e-15&   437 &&  1.144e04  & 1.068e04 & 50.7 &  8e-15 &  385  \\
 100 &  1.065e04  & 1.030e04 & 25289.8 &  1e-07 &&   1.200e04  & 1.029e04 &  127.9 &  9e-15&   785 &&  1.200e04  & 1.028e04 & 67.9 &  8e-15 &  409  \\
 200 &  1.032e04  & 1.022e04 & 4350.6 &  5e-07 &&   1.824e04  & 1.023e04 &  178.4 &  1e-14&   896 &&  1.824e04  & 1.022e04 & 77.8 &  8e-15 &  387  \\
 400 &  1.032e04  & 1.022e04 & 303.6 &  1e-07 &&   2.920e04  & 1.023e04 &  344.4 &  1e-14&  1135 &&  2.920e04  & 1.022e04 & 116.5 &  9e-15 &  385  \\  	
 \Xhline{0.6pt}
  \end{tabular}}
\end{scriptsize}
\end{table}

\rev{Here we should notice that PenCorr is a general and powerful package since it can deal with more
general constraints such as lower and upper bound constraints.  While it is not known yet how to extend AFBB efficiently in this situation,
we shall consider a possible extension of the problem \reff{prob:NLCM} with some given entries of the matrix $X$. More exactly, it is required that
$X_{ij} = q_{ij}$ for $(i,j)\in \mathcal{B}_e$, where $\mathcal{B}_e$ is the subset of $\{(i,j) | 1 \leq j < i \leq n\}$ satisfying
$-1 \leq \rrev{q}_{ij} \leq 1$ for any $(i,j) \in \mathcal{B}_e $. Using the decomposition $X = V^{\msT}V$, the problem is now equivalent to
\be \label{prob:extNLCM2}
\begin{array}{cl}
  \min\limits_{V \in \mbR^{r \times n}} &\theta(V; H, C)  \\
\mbox{s.t.} &  \|V_i\|_2 = 1,\, i = 1, \cdots, n, \\
& V_i^{\msT}V_j - q_{ij}  = 0, \ (i,j) \in \mathcal{B}_e.
\end{array}
\ee
To deal with the extra nonlinear constraints, we introduce the augmented Lagrangian function (see \cite{andreani2007augmented, nocedal2006numerical, sunyuan2006optimization}) of \reff{prob:extNLCM2} as follows:
\be
    \rrev{L_{\mu}(V, \Lambda) =
    \theta(V; H, C) + \frac{\mu}{2}\theta\big(V; H_e, \widehat{C} + \Lambda/\mu\big)}
    \nn
 \ee
where $\mu>0$ is the penalty parameter,  $\Lambda$, $H_e$ and $\widehat{C}$ are matrices in $\mbR^{n\times n}$ with zero entries for all $(i,j)\notin \mathcal{B}_e$.
Starting from the initial point  $C_{\mathrm{pca}}$ or $\widetilde{C}_{\mathrm{pca}}$, $\Lambda_0 = 0$,  \revmini{$\mu_0 = 10$},  the procedure  of the augmented Lagrangian method is  as follows:
 \be\label{equ:agnlnlcm}
 \left\{
 \begin{array}{ccl}
   V_{k+1} &:\approx& \argmin\limits_{V \in \mbR^{r \times n}} \  L_{\mu_k}(V, \Lambda_k), \ \ \mbox{s.t.}     \ \  \|V_i\|_2 = 1,\, i = 1, \cdots,      n, \\
   \Lambda_{k+1}&:= &\Lambda_k - \mu_k H_e \odot  (V_k^{\msT}V_k^{} - \rrev{\widehat{C}_k}), \\
 \mu_{k+1} &:= &10 \mu_k.
 \end{array}
 \right.
 \ee
 The subproblem in \reff{equ:agnlnlcm} is a low-rank nearest correlation matrix problem and is solved inexactly by the AFBB method. Specifically, for the $k$-th subproblem, the main parameters of AFBB are set to be
 $$\epsilon_k = \max\{0.1\epsilon_{k-1}, 10^{-5}\}, \, \epsilon_{x,k} = \max\{0.1\epsilon_{x, k-1}, 10^{-5}\}, \, \epsilon_{f,k} = \max\{0.1\epsilon_{f,k-1}, 10^{-8}\}, \, \text{MaxIter}_k = 2000,$$
 where $\epsilon_0 = 10^{-1}, \epsilon_{x,0} = 10^{-3}, \epsilon_{f,0} = 10^{-5}.$ Denoting  $\nu_k = \sum_{(i,j) \in \mathcal{B}_e} |V_{k,i}^{\msT} \rrev{V_{k,j}^{}} - q_{i,j}|$, we terminate the procedure \reff{equ:agnlnlcm} when $\nu_{k+1} \leq 3\times 10^{-8}$ or $|\nu_{k+1}  - \nu_k| \leq 10^{-8}$. }

Below we consider a test instance of problem \reff{prob:NLCM2} with extra fixed constraints.

{Ex. 10}: The matrices  $C, H$  are  the same as the ones in {Ex. 4}. The index set $\mathcal{B}_e$ consists of $\min\{n_e, n - i\}$ randomly generated  integers from $\{1, \ldots, n\}$ with $n_e = 3$. We set $q_{ij} = 0$ for $(i,j) \in \mathcal{B}_e$.

\rev{The numerical results are presented in Table \ref{table:extNLCM}, where ``const.'' represents the total constraint violation $\nu_k$. From the table,
we can see that the AFBB method is quite promising, which shows its potential to handle some more general constraints beyond the sphere constraints. As the augment Lagrangian method belongs to a different class of methods, however, we shall go further on this topic elsewhere. }

\begin{table}[!htbp]
\linespread{1.5}
     \centering
     \revmini{
\caption{Numerical results of the nearest low-rank correlation matrix problem with equality requirements: PenCorr, AFBB} \label{table:extNLCM}
~\\[0.6pt]
\begin{scriptsize}
  \begin{tabular}{@{}l r r r r r  r r r r r @{}}
   \Xhline{0.6pt}
   \Gape[8pt] &  \multicolumn{4}{c}{PenCorr} &&  \multicolumn{5}{c}{AFBB} \\
  \Xcline{2-5}{0.6pt} \Xcline{7-11}{0.6pt}
  \Gape[8pt] r & residual  &  time & feasi&    const. &&  residual & time & feasi & const. & nfge \\
\Xhline{1.1pt}
\multicolumn{11}{@{}l@{}}{{Ex. 10}, $n = 943$,  $H = \textbf{1}$} \\
10 & 3.533e02 & 1082.1 & 2e-08 &  2e-07 &&   3.529e02 &  73.1 & 6e-15 &  2e-08  &2678 \\
50 & 2.841e02 &  93.8 & 4e-08 &  9e-08  &&   2.841e02 &  15.1 & 6e-15 &  5e-08  &446 \\
100 & 2.827e02 &  36.3 & 1e-07 &  7e-08 &&   2.827e02 &  16.9 & 7e-15 &  6e-09  &398 \\
200 & 2.827e02 &  23.4 & 9e-09 &  6e-09 &&   2.827e02 &  24.0 & 7e-15 &  2e-08  &379 \\
250 & 2.827e02 &  23.5 & 9e-09 &  6e-09 &&   2.827e02 &  23.6 & 7e-15 &  2e-08  &355 \\
\multicolumn{11}{@{}l@{}}{{Ex. 10}, $n = 943, H$ given by T. Fushiki}  \\
10 & 8.306e03 & 5292.7 & 3e-08 &  5e-07 &&   8.238e03 & 200.9 & 6e-15 &  1e-08  &6221 \\
50 & 4.171e03 & 1540.5 & 2e-08 &  1e-08 &&   4.165e03 &  63.5 & 7e-15 &  3e-09  &1658 \\
100 & 3.939e03 & 1012.4 & 1e-08 &  9e-09 &&   3.936e03 &  55.0 & 7e-15 &  1e-08  &1189 \\
200 & 3.936e03 & 905.5 & 2e-08 &  1e-08 &&   3.933e03 &  60.4 & 7e-15 &  2e-08  &1154 \\
250 & 3.936e03 & 917.7 & 2e-08 &  1e-08 &&   3.933e03 &  67.1 & 8e-15 &  2e-08  &1148 \\
\Xhline{0.6pt}
\end{tabular}
\end{scriptsize}
}
\end{table}

\subsection{Kohn-Sham total energy minimization}
We test in this subsection a class of nonlinear eigenvalue problems known as Kohn-Sham (KS) equations, which arise in electronic structure calculations. The original KS equation or KS energy minimization problem is a continuous nonlinear problem. To solve the problems numerically, we  turn them into finite-dimensional problems. Let the unitary  matrix $X \in \mathbb{C}^{n\times p}$ be the approximation of the electronic wave functions of $p$ occupied states, where $n$ is the spatial degrees of freedom. Using the planewave basis, we define the finite-dimensional approximation to the total energy functional as follows:
\be\label{eqution:totalKS}
E_{\text{total}}(X) = \mtr\bigg(X^*\Big(\frac{1}{2}\widehat{L} + V_{\text{ion}}\Big)X\bigg) + \frac{1}{2} {\rho(X)}^{\msT} \widehat{L}^{\dagger} \rho(X) + \rho(X)^{\msT} \epsilon_{\mathrm{xc}}({\rho(X)}) + E_{\text{Ewald}} + E_{\text{rep}},
\ee
where $\widehat{L}$ is a finite-dimensional representation of the Laplacian operator in the planewave basis, $V_{\text{ion}}$ denotes the ionic pseudopotentials sampled on the suitably chosen Cartesian grid, ${\rho(X)} = \text{diag}(XX^*)$ represents  the charge density, the matrix  $\widehat{L}^{\dagger}$ means the pseudo-inverse of $\widehat{L}$, and $\epsilon_{\mathrm{xc}}({\rho(X)})$ is  the exchange-correlation energy per particle in a uniform electron gas of density ${\rho(X)}$.
The last two terms in \reff{eqution:totalKS} are constant shifts of the total energy (see \cite{yang2009kssolv} for the details). Consequently, the discretized KS energy minimization can be formulated as
\be \label{prob:KS}
\min_{X \in \mathbb{C}^{n \times p}} E_{\text{total}}(X) ,\ \ \mbox{s.t.} \ \   X^*X=I_p.
\ee
For the above problem, the KKT conditions; {\it i.e.,} the KS equations, are
\be \label{prob:ksequ}
  H(X) X - X \Lambda_p = 0, \ \  X^*X = I_p.
\ee
Here, $H(X) = \frac{1}{2} \widehat{L} + V_{\text{ion}} + \text{Diag}(\widehat{L}^{\dagger} {\rho(X)}) + \text{Diag}(\mu_{\mathrm{xc}}({\rho(X)}))$, $\mu_{\mathrm{xc}}({\rho(X)}) = d \mu_{\mathrm{xc}}({\rho(X)})/d {\rho(X)}$ and $\Lambda_p$ is a $p$-by-$p$ symmetric matrix of Lagrangian multipliers.

Generically, it is sophisticated to compute the objective function and its gradient in \reff{prob:KS}. Thus, we use the Matlab Toolbox KSSOLV \cite{yang2009kssolv} which is tailored for easily developing new algorithms for solving the KS equations to do so. In order to show the efficiency of our AFBB method, we compared it with the self-consistent field (SCF) iteration which is currently the most widely used approach for the KS equations, this SCF iteration is provided  in KSSOLV. We also compared AFBB with OptM \cite{wen2013feasible} for KS problem \reff{prob:KS} but only the perfomance of AFBB is reported because they performed similarly. The maximal iteration of SCF was set to be $200$ while the other parameters were set to be  their default values in KSSOLV. For the sake of fairness, we improved the stopping accuracy of  AFBB; {\it i.e.}, resetting $\epsilon = 10^{-6}, \epsilon_x = 10^{-10}, \epsilon_f = 10^{-14}, \text{MaxIter} = 1000$, to obtain a higher quality solution.  The termination rules are not directly comparable due to the different formulations of the problem used by SCF and AFBB. Specifically, SCF focuses on  KS equations  \reff{prob:ksequ} and needs to solve a series of linear eigenvalue problems, while AFBB minimizes the total energy directly.  However, as shown later  by the numerical results in Table \ref{table:KS}, we see that  on average the chosen stopping criteria  for AFBB are tighter than those of SCF in terms of the residual $\|HX - H(X^*HX)\|_{\msF}^{\hspace{0.1mm}}$.  For each problem, we ran the two algorithms $10$ times from different random initial points generated by the function ``genX0'' provided in KSSOLV. \rev{Note that for each instance, AFBB and SCF use the same initial point.}

\begin{table}[!htbp]
\linespread{1.5}
  \centering
  \caption{Numerical results of the Kohn-Sham total energy minimization: SCF, AFBB}\label{table:KS}
  ~\\
  \begin{scriptsize}
    \revmini{  \begin{tabular}{@{\hspace{0.1mm}}  lcccrccr @{\hspace{0.1mm}} }
      \Xhline{0.6pt}
      \Gape[8pt] solver&  a.$E_{\text{total}}^0$ &  a.$E_{\text{total}}$  & a.resi  &a.iter&a.feasi&a.err&a.time\\
      \Xhline{1.1pt}
  \multicolumn{8}{@{\hspace{0.1mm}}l@{\hspace{0.1mm}}}{   alanine, $n = 12671, p =   18$} \\
SCF & -6.078e01 &-6.116e01  & 3e-06 & 48.7   &2e-14   &  2e-13&  174.5  \\
AFBB & -6.078e01 & -6.116e01    & 4e-07 & 77.8   &5e-15  &  0e00  &  137.8  \\ 
\multicolumn{8}{@{\hspace{0.1mm}}l@{\hspace{0.1mm}}}{        al, $n = 16879, p =   12$} \\
SCF & -1.576e01 &  -1.577e01  & 8e-03 & 200.0   &6e-15   &  8e-04& 1490.6  \\
AFBB & -1.576e01 &-1.580e01    & 7e-05 & 986.5   &6e-13  &  0e00  & 1212.1  \\ 
\multicolumn{8}{@{\hspace{0.1mm}}l@{\hspace{0.1mm}}}{   benzene, $n = 8407, p =   15$} \\
SCF & -3.693e01 &  -3.723e01  & 2e-06 & 35.1   &1e-13   &  8e-14&   70.3  \\
AFBB & -3.693e01 &-3.723e01    & 4e-07 & 70.4   &4e-15  &  0e00  &   56.5  \\ 
  \multicolumn{8}{@{\hspace{0.1mm}}l@{\hspace{0.1mm}}}{    c12h26, $n = 5709, p =   37$} \\
SCF & -8.073e01 &-8.154e01    & 3e-06 & 69.1   &4e-14  &  2e-13  &  216.6  \\
AFBB & -8.073e01 &-8.154e01    & 6e-07 & 78.9   &1e-14  &  0e00  &  129.7  \\ 
\multicolumn{8}{@{\hspace{0.1mm}}l@{\hspace{0.1mm}}}{  ctube661, $n = 12599, p =   48$} \\
SCF & -1.340e02 & -1.346e02  & 3e-06 & 50.9   &4e-14   &  5e-14&  535.6  \\
AFBB & -1.340e02 &-1.346e02    & 6e-07 & 84.6   &8e-14  &  0e00  &  377.7  \\ 
\multicolumn{8}{@{\hspace{0.1mm}}l@{\hspace{0.1mm}}}{ glutamine, $n = 16517, p =   29$} \\
SCF & -9.087e01 & -9.184e01  & 3e-06 & 49.2   &3e-14   &  8e-14&  436.8  \\
AFBB & -9.087e01 &-9.184e01    & 7e-07 & 96.2   &9e-14  &  0e00  &  373.4  \\ 
\multicolumn{8}{@{\hspace{0.1mm}}l@{\hspace{0.1mm}}}{graphene16, $n = 3071, p =   37$} \\
SCF & -9.358e01 & -9.400e01  & 6e-03 & 200.0   &2e-14   &  5e-04& 1058.8  \\
AFBB & -9.358e01 &-9.405e01    & 4e-06 & 219.5   &4e-13  &  0e00  &  176.2  \\ 
\multicolumn{8}{@{\hspace{0.1mm}}l@{\hspace{0.1mm}}}{graphene30, $n = 12279, p =   67$} \\
SCF & -1.726e+02 & -1.735e+02  & 9e-03 & 200.0   &3e-14   &  3e-04& 9225.4  \\
AFBB & -1.726e+02 &-1.736e+02    & 6e-07 & 294.2   &2e-14  &  0e00  & 2007.3  \\ 
\multicolumn{8}{@{\hspace{0.1mm}}l@{\hspace{0.1mm}}}{pentacene, $n =  44791, p =    51$} \\
SCF & -1.311e+02 & -1.319e+02  & 4e-06 & 65.9   &5e-14   &  5e-12& 3253.1  \\
AFBB & -1.311e+02 &-1.319e+02    & 7e-07 & 118.7   &2e-14  &  0e00  & 2046.5  \\ 
  \multicolumn{8}{@{\hspace{0.1mm}}l@{\hspace{0.1mm}}}{     ptnio, $n = 4069, p =   43$} \\
SCF & -1.983e02 & -2.268e02  & 8e-06 & 200.0   &2e-14   &  2e-09&  946.7  \\
AFBB & -1.983e02 &  -2.268e02    & 4e-06 & 459.6   &1e-14  &  0e00  &  480.2  \\ 
  \multicolumn{8}{@{\hspace{0.1mm}}l@{\hspace{0.1mm}}}{      qdot, $n = 2103, p =    8$} \\
SCF & 2.850e01 & 2.771e01  & 2e-02 & 200.0   &4e-15   &  2e-04&  106.5  \\
AFBB & 2.850e01 &2.770e01   & 3e-04 & 1000.0   &3e-15  &  0e00  &   81.2  \\ 
\Xhline{0.6pt}
 \end{tabular}
}
\end{scriptsize}
\end{table}

A  summary of the numerical results \revmini{on $11$ standard testing problems} is reported in Table \ref{table:KS}. In this table, \rev{``a.$E_{\text{total}^{}}^0$'' and ``a.$E_{\text{total}}$''} \rev{represent the average initial and returned} total energy function value, \brev{respectively.}  \brev{The term \brev{``a.iter''} denotes the average} total number of iterations, \brev{``a.resi''}, \brev{``a.feasi''}  and  \brev{``a.time''} the average residual  $\|HX  - X (X^*HX)\|_{\msF}^{}$, the average violation of the constraint $X^*X = I_p$ and  the average CPU time in seconds, respectively. We use \brev{``a.err''} to denote  the average relative errors between the average total energy $\bar{z}_{1}$ given by AFBB  or the  average total energy $\bar{z}_{2}$ given by SCF and the minimal of $\bar{z}_1$ and $\bar{z}_2$, which is computed by $\frac{\bar{z}_1 - \min\{\bar{z}_1, \bar{z}_2\}}{|\min\{\bar{z}_1, \bar{z}_2\}|}$ or $\frac{\bar{z}_2 - \min\{\bar{z}_1, \bar{z}_2\}}{|\min\{\bar{z}_1, \bar{z}_2\}|}$.  From the table, we can see that the AFBB method is considerably competitive and it can always take less CPU time than SCF to find the better solutions in terms of the total energy and the residual, especially for the large molecules. In particular, for the most hard problem ``graphene30'' in our test, AFBB is not only  significantly faster than SCF, but also  returns a better solution with smaller total energy and residual.

\subsection{Maximization of sums of heterogeneous quadratic functions on the Stiefel manifold from statistics}\label{subsection:hetequadratic}
In \cite{Balogh2004some}, Balogh {\it et al.} gave a test problem with known global optimal solutions  as follows:
    \be \label{prob:knowproblem}
    \min_{X \in \mathbb{R}^{n \times p}}\, \sum_{i=1}^{p}X_{(i)}^{\msT}A_i^{} X_{(i)}, \ \  \mbox{s.t.} \ \  X^{\msT}X=I_p,
     \ee
     where for $1 \leq i \leq p$,  $A_i = \mbox{\brev{Diag}}\big(n(i-1)+1,\, \ldots,\,n(i-1)+i-1,\, l_i,\,n(i-1)+i+1,\,\ldots,\,ni\big)$
       and $l_i  < 0$.  This is a special  maximization of sums of heterogeneous quadratic functions on the Stiefel manifold from statistics
       (it was also considered in \cite{bolla1998extrema}).
     By Proposition 1 in \cite{Balogh2004some}, we know that
       $ \{(\pm e_1, \pm e_2, \ldots, \pm e_p):  \pm e_i \in \{e_i, -e_i\}\}$
is the set of minimum points of  problem (\ref{prob:knowproblem}) and  $\sum_{i=1}^k l_i$ is the optimal function value.   It was pointed out in \cite{Balogh2004some} that  there were no efficient numerical methods to solve problem \reff{prob:knowproblem} yet. Nevertheless, our numerical tests show that the AFBB method works well.

In this experiment, we reset $\epsilon = 10^{-6}, \epsilon_x = 10^{-6}, \epsilon_f = 10^{-10}$ and  fixed  $n=4000$. For $1\leq p\leq n$ and $1\leq i \leq p$, we  generated $l_i$ in two ways, one is that  $l_i$ is uniformly distributed in $[0,\,1]$, the other is that  $l_i = -1$. We ran our AFBB  methods  $50$ times from different random initial points for each test.

 Firstly, we investigate the effect of using different descent directions. We call AFBB methods using the  descent directions  $D_{1/2}$ and $D_{1/4}$ as AFBB$D_{1/2}$ and AFBB$D_{1/4}$, respectively. The numerical results  are shown in Table \ref{table:knownproblem}. In this table,
 \rev{``a.obj$^0$'' and ``a.obj'' represent the average initial funciton value and returned function value by each method, respectively.  The terms $f^*$, ``a.nfe'' and ``a.err''} denote the optimal function value, the average total number of function evaluations and  the average relative error between the function value given by each method and $f^*$, respectively. We use  \brev{``a.s.ratio''} to stand for  the average saved ratio of AFBB$D_{1/4}$ which  is computed as  $100 (\mbox{nfe}_{\rho = 1/4} - \mbox{nfe}_{\rho = 1/2})/\mbox{nfe}_{\rho = 1/2}$. From this table, we know that the two methods can always find nearly global solutions in acceptable iterations. Averagely, AFBB$D_{1/4}$  can always find a better solution with smaller function value about 25\% faster than AFBB$D_{1/2}$ for most of the tests. This  may be  due to the fact that $D_{1/4}$ is the steepest descent direction corresponding to the \emph{Euclidean metric}, that is,
$$D_{1/4} = {\argmin\limits_{D \in \mT_X} \, -\frac{\langle G, D\rangle}{\|D\|_{\msF}^{}}. }$$

\begin{table}[!tbp]
\linespread{1.5}
  \centering
  \caption{Numerical results of AFBB$D_{1/2}$ and AFBB$D_{1/4}$ for  problem \reff{prob:knowproblem}} \label{table:knownproblem}
  ~\\
  \begin{scriptsize}
\revmini{
\begin{tabular}{@{}l@{\hspace{2.8mm}}c@{\hspace{2.8mm}}c@{\hspace{1.8mm}}c@{\hspace{3mm}}c@{\hspace{2.8mm}}c@{\hspace{2.8mm}}c@{\hspace{.5mm}}c@{\hspace{6mm}}c@{\hspace{2.8mm}}c@{\hspace{2.8mm}}c@{\hspace{2.8mm}}r@{}}
\Xhline{0.6pt}
\Gape[8pt] & \multicolumn{3}{c}{} &\multicolumn{3}{c}{AFBB$D_{1/2}$} &&\multicolumn{4}{c}{AFBB$D_{1/4}$} \\
\Xcline{5-7}{0.6pt} \Xcline{9-12}{0.6pt}
\Gape[8pt] p &a.obj$^0$ &$f^*$ &&  a.obj & a.nfe & a.err &&  a.obj & a.nfe & a.err & a.s.ratio \\ \Xhline{1.1pt}
\multicolumn{12}{@{}l@{}}{random $l$} \\
2  & 8.003e03 &  -1.304e00 &&  -1.304e00  &  403.9 &  4e-07  &&   -1.304e00  &   405.9  &  4e-07  &  0.5  \\
20  & 7.998e05 &  -8.902e00 &&  -8.902e00  &  839.8 &  7e-07  &&   -8.902e00  &   615.4  &  4e-07  & -26.7  \\
60  & 7.198e06 &  -3.096e01 &&  -3.096e01  &  962.8 &  8e-07  &&   -3.096e01  &   692.0  &  4e-07  & -28.1  \\
100  & 2.000e07 &  -5.087e01 &&  -5.087e01  &  1032.5 &  1e-06  &&   -5.087e01  &   714.9  &  4e-07  & -30.8  \\  
\multicolumn{12}{@{}l@{}}{$l_i = -1, i = 1, \ldots,p$} \\
2  & 7.997e03 &  -2.000e00 &&  -2.000e00  &  391.1 &  3e-07  &&   -2.000e00  &   397.6  &  2e-07  &  1.7  \\
20  & 7.998e05 &  -2.000e01 &&  -2.000e01  &  822.7 &  1e-06  &&   -2.000e01  &   597.2  &  4e-07  & -27.4  \\
60  & 7.198e06 &  -6.000e01 &&  -6.000e01  &  890.7 &  4e-06  &&   -6.000e01  &   645.6  &  4e-07  & -27.5  \\
100  & 2.000e07 &  -1.000e02 &&  -1.000e02  &  942.0 &  8e-07  &&   -1.000e02  &   696.2  &  4e-07  & -26.1  \\
 \Xhline{0.6pt}
  \end{tabular}}
\end{scriptsize}
\end{table}

Secondly, we consider the effect of using different {functions} $g(\tau)$ in forming $J(\tau)$. \rev{Here, we choose $\rho = 0.5$.} We tested two choices: $g_1(\tau) = \tau/2$  which is the default in update scheme \reff{equ:update:new} and $g_2(\tau) = \frac{1}{2}\tau e^{-\tau}$. We call AFBB methods using $g_1(\tau)$ and $g_2(\tau)$ as AFBBg1 and AFBBg2, respectively.  From Table \ref{table:knownproblem:gtau}, we see that  AFBBg2 improves the performance of AFBBg1  by 15\% or more for most tests in terms of  the average total number of function evaluations. Meanwhile, it  can always  return a soultion with smaller function value. Nevertheless, it remains under investigation how to seek a  better $g(\tau)$.

At the end of this subsection, we remark that $X^{\msT}G \equiv G^{\msT}X$ happens  in  the nearest low-rank correlation matrix problem and the Kohn-Sham total energy minimization. In this case, the term $X^{\msT}D_{\rho}$ vanishes and  $g(\tau)$ will not play a role in forming $J(\tau)$.

  \begin{table}[!ht]
\linespread{1.5}
  \centering
  \caption{Numerical results of AFBBg1 and AFBBg2 for  problem \reff{prob:knowproblem}} \label{table:knownproblem:gtau}
  ~\\
  \revmini{
  \begin{scriptsize}
    \begin{tabular}{@{}l@{\hspace{2.8mm}}c@{\hspace{2.8mm}}c@{\hspace{1.8mm}}c@{\hspace{3mm}}c@{\hspace{2.8mm}}c@{\hspace{2.8mm}}c@{\hspace{.5mm}}c@{\hspace{6mm}}c@{\hspace{2.8mm}}c@{\hspace{2.8mm}}c@{\hspace{2.8mm}}r@{} }
      \Xhline{0.6pt}
 \Gape[8pt] &\multicolumn{3}{c}{} &\multicolumn{3}{c}{AFBBg1} &&\multicolumn{4}{c}{AFBBg2} \\
    \Xcline{5-7}{0.6pt} \Xcline{9-12}{0.6pt}
    \Gape[8pt] p &a.obj$^0$ &$f^*$ &&  a.obj & a.nfe & a.err &&  a.obj & a.nfe & a.err & a.s.ratio \\ \Xhline{1.1pt}
    \multicolumn{12}{@{} l @{} }{random $l$} \\
 2  &  8.004e03  &  -5.653e-01  & &  -5.653e-01  &  409.1 &  1e-06  &&   -5.653e-01  &   436.4  &  7e-07  &  6.7  \\
  20  &  7.998e05  &  -1.054e01  & &  -1.054e01  &  839.7 &  7e-07  &&   -1.054e01  &   710.3  &  3e-07  & -15.4  \\
  60  &  7.198e06  &  -2.790e01  & &  -2.790e01  &  941.3 &  8e-07  &&   -2.790e01  &   789.1  &  6e-07  & -16.2  \\
100  &  2.000e07  &  -4.158e01  & &  -4.158e01  &  1039.8 &  8e-07  &&   -4.158e01  &   856.0  &  4e-07  & -17.7  \\
\multicolumn{12}{@{} l @{} }{$l_i = -1, i = 1, \ldots,p$} \\
2  &  7.998e03  &  -2.000e00  & &  -2.000e00  &  401.8 &  2e-07  &&   -2.000e00  &   402.2  &  3e-07  &  0.1  \\
20  &  7.998e05  &  -2.000e01  & &  -2.000e01  &  816.5 &  6e-07  &&   -2.000e01  &   717.6  &  3e-07  & -12.1  \\
60  &  7.198e06  &  -6.000e01  & &  -6.000e01  &  931.1 &  1e-06  &&   -6.000e01  &   741.4  &  3e-07  & -20.4  \\
100  &  2.000e07  &  -1.000e02  & &  -1.000e02  &  987.5 &  2e-06  &&   -1.000e02  &   869.4  &  3e-07  & -12.0  \\
 \Xhline{0.6pt}
  \end{tabular}
\end{scriptsize}
}
\end{table}

\section{Conclusion}
In this paper, we have proposed a  feasible method for optimization on the Stiefel manifold. Our main contributions are twofold.
Firstly, we proposed a new framework of constraint preserving update schemes for optimization on the  Stiefel manifold by decomposing each feasible point into the range space of $X$ and the null space of $X^{\msT}$. While this new framework can unify many existing schemes, we also investigated a new update scheme with low complexity. Note that our framework can be viewed as a retraction as well. Secondly, we proposed the adaptive feasible Barzilai-Borwein-like method and proved  its global convergence. To our knowledge,  this result is the  first global convergence result for the feasible  method with nonmonotone linesearch for optimization on the  Stiefel manifold. Moreover, the corresponding extension to the  generalized Stiefel manifold was also considered.

We have tested our AFBB method  on a variety of problems  to illustrate its efficiency. Particularly, for  the nearest low-rank correlation matrix problem, AFBB performs better than \lastrev{three} state-of-the-art algorithms. \rev{Note that \lastrev{PenCorr}, one of the three algorithms, can deal with more general problems.} For Kohn-Sham total energy minimization,  \lastrev{the superiority of AFBB} is obvious especially for large molecules, and hence it is quite promising to use our AFBB for large-scale electronic structure calculations. Since our update scheme \lastrev{is compatible with} moving along any given tangent direction, we  also explore the effect of different descent directions and different $g(\tau)$'s on the performance of AFBB.

As our framework can unify several famous retractions, it is natural and  interesting to argue  which one can make the AFBB method find the global optimal solution with highest probability and at the fastest speed. One possible approach is to consider  the subspace techniques. This remains under investigation.
\section*{Acknowledgements}
{ \wuhao
  \renewcommand\baselinestretch{1.0}\selectfont
  We are very grateful to Zaiwen Wen for useful discussions on the optimization on the Stiefel manifold and  Qingna Li for discussions on the nearest  low-rank correlation matrix problem. We thank Zaiwen Wen and Wotao Yin for sharing their code FOptM,  Houduo Qi and  Qingna Li for sharing their code SemiNewton, Defeng Sun and  Yan Gao for sharing their code PenCorr.  We also thank  Kaifeng Jiang, Defeng Sun and  Kim-Chuan Toh for sharing the gene correlation matrices for the nearest correlation matrix problem. \rev{Many thanks are also due to the \revmini{associate editor},  two anonymous referees and the editor, Prof. Alexander Shapiro, whose suggestions and comments greatly improved the quality of this paper.}
  \begin{appendices}
    \section{Details of Approach I and II in \S 3}\label{appendix:section:approach12}
 \subsection{Details of Approach I}\label{appendix:subsection:approach1}
 \revmini{Using the condition that $Z(\tau) \equiv I_p$ and $\wR'(0) = - X^{\msT} E$ and denoting} $\mA(\tau) = \left[\begin{array}{c} \wR(\tau) \\ W \wN(\tau) \end{array}\right]$, \revmini{it follows from } \reff{equ:condition1} and \reff{equ:condition2} that
  \be
  \mA'(0) = \left[\begin{array}{cc} -X^{\msT}E &~B \\ W & ~F \end{array}\right] \left[\ba{c} I_p \\ 0\ea\right]\ \mathrm{and} \ \mA(0) = \left[\ba{c} I_p \\ 0\ea\right], \nn
  \ee
  where $B \in \mbR^{n\times p}$, $F \in \mbR^{p \times p}$. Solving the above ordinary differential equations, we get that
  \be \label{equ:update:geodesic:frame1}
  \mA(\tau) = \exp\left(\tau\left[\begin{array}{cc} -X^{\msT}E &~B \\ W & ~F \end{array}\right]\right) \left[\ba{c} I_p \\ 0\ea\right].
  \ee
  Since $Z(\tau) \equiv I_p$, we know from \reff{equ:ztau2} and the definition of $\mA(\tau)$ that
  \be
  Y(\tau) = \big[\ba{l} X, \ I_p\ea \big]  \mA(\tau), \nn
  \ee
 and $\mA(\tau)^{\msT} \mA(\tau) = I_p$. It follows that the matrix inside the brackets of \reff{equ:update:geodesic:frame1} will be skew-symmetric. This means that $B = -W^{\msT}$,  $F + F^{\msT} = 0$ and $X^{\msT}E+E^{\msT}X=0$
  ({\it i.e.}, $E \in \mT_X$). Let $W = QR$ be the unique QR factorization of $W$ and assume that $F$ takes the form $F = Q\wF Q^{\msT}$, where $\wF \in \mbR^{p\times p}$ is any skew-symmetric matrix. Then we can write
  \begin{align}\label{equ:update:geodesic:frame2}
    Y(\tau)  ={} &\big[\ba{l} X, \ I_p\ea \big]  \exp\left( \tau\left[\begin{array}{cc} I_p &~0 \\ 0 & ~Q \end{array}\right] \left[\begin{array}{cc} -X^{\msT}E &~-R^{\msT} \\ R & ~\wF \end{array}\right]  \left[\begin{array}{cc} I_p &~0 \\ 0 & ~Q^{\msT} \end{array}\right] \right)\left[\ba{c} I_p \\ 0\ea\right]  \nn \\
      ={} & \big[\ba{l} X, \ Q \ea \big]  \exp\left(\tau\left[\begin{array}{cc} -X^{\msT}E &~-R^{\msT} \\ R & ~\wF \end{array}\right]\right)  \left[\ba{c} I_p \\ 0\ea\right],
\end{align}
where the second equality is due that $Q$ has orthonormal columns. The update scheme \reff{equ:update:geodesic:frame2} can be regarded as a generalized geodesic update scheme, since letting $\wF = 0$, \reff{equ:update:geodesic:frame2} reduces to the geodesic update scheme \reff{equ:update:geodesic}.

\subsection{Details of Approach II}\label{appendix:subsection:approach2}
 \revmini{Based on the choices that $\wR(\tau) = I_p + \tau \wR'(0)$ and  $\wN(\tau) = \tau I_p$, we can get $Z(\tau)$ from \reff{equ:ztau2} by the  polar decomposition or the Cholesky factorization.}

    If by the  polar decomposition, $Z(\tau)$ and $Z'(\tau)$ are always symmetric. In this case, it is easy to show that
 $Y(\tau; X) = \mP_{\stief}\rev{\big(X \wR(\tau) + W \wN(\tau) \big)}$, 
 which with \reff{equ:condition2}, \rrev{\reff{add3} and  \reff{equ:approach2:wR}} further implies  that
 \be \label{equ:generalpolar}
   Y(\tau; X) = \mP_{\stief}\big(X - \tau E -  \tau X Z'(0) \big),
 \ee
 where $Z'(0)$ is any $p$-by-$p$ symmetric matrix. If $Z'(0) = 0$, \reff{equ:generalpolar} reduces to the ordinary polar decomposition or Manton's projection update scheme. If $Z'(0) = \mbox{sym}(X^{\msT}G)$ and $E = G - X \mbox{sym}(X^{\msT}G),$ it becomes the ordinary gradient projection update scheme.

 If by the Cholesky factorization, $Z(\tau)$ and $Z'(\tau)$ are always upper triangular. Similarly, we can derive
  \be \label{equ:generalqr}
   Y(\tau; X) = \mbox{qr}\big(X - \tau E -  \tau X Z'(0)\big),
 \ee
 where $Z'(0)$ is any $p$-by-$p$ upper triangular matrix. When $Z'(0) = 0$, \reff{equ:generalqr} reduces to the ordinary QR factorization update scheme.

    \section{Proof of Lemma \ref{lemma:update}}\label{appendix:section:proof1}
  The fact that  $X^{\msT}\Drho$  is  skew-symmetric implies  that $J$ is  invertible, \rrev{thus $Y(\tau)$ is well-defined.} (i) follows from the construction of the update scheme \reff{equ:update:new}.

Meanwhile, we know that $Y'(0) = -\Drho$, \rev{which with the chain rule shows that
  \be
 \mF'_{\tau}(Y(0)) = -  \langle G, \Drho \rangle  = -\langle \nabla \mF, P_X \Drho \rangle, \nn
  \ee
  where the second  equality is due to $\Drho \in \mT_X$  and  the definition \reff{equ:partialStiefel} of $\nabla \mF$. \rrev{Recall that $P_X = I_n - \frac12 XX^{\msT}$.}  Substituting \reff{equ:Drho:def} into the above equation yields
  \be
  \mF'_{\tau}(Y(0))
    = -\langle \nF, (I_n + (\rho -1 )XX^{\msT})\nF \rangle
    \leq - \min\{\rho,1\} \|\nF\|_{\msF}^2. \nn
  \ee
}
So (ii) is true.

We prove (iii) by contradiction.  \rev{Multiplying $X^{\msT}$ from both sides the expression of $Y(\tau)$ in \reff{equ:update:new} and using $X^{\msT}W = 0$, we get that $2J^{-1} - I_p = X^{\msT}Y$.} Assume that there exists a $p$-by-$p$ orthogonal matrix $Q_p$ such that $Y = XQ_p$.  Then we have $2J^{-1} - I_p = Q_p$; {\it i.e.}, $2I_p - J = Q_pJ$. It follows from $(2I_p - J)^{\msT} (2I_p-J) = (Q_pJ)^{\msT}Q_pJ$ and  $Q_p^\msT Q_p = I_p$ that $$J^{\msT} + J = 4I_p,$$ which is a contradiction to the definition of $J$. The contradiction shows the truth of (iii).

For (iv), it is obvious that  $\|J\|_2^{} \leq 1 + \frac{\tau^2}{4}\|W\|_{\msF}^2 + \frac{\tau}{2}\|X^{\msT}\Drho\|_{\msF}^{}$, where $\|\cdot\|_2^{}$ is the spectral norm.
\rev{It follows from  $-\Drho = W - XX^{\msT} \Drho$ and $X^{\msT} W=0$ that  $\|\Drho\|_{\msF}^2 = \|W\|_{\msF}^2 + \|X^{\msT}\Drho\|_{\msF}^2.$ \rrev{Notice that $\upsilon = \tau \|\Drho\|_{\msF}^{}$.} Then we have
\begin{align}
 \|J\|_2  &\leq  1 + \frac{\upsilon^2}{4} \cdot \frac{\|W\|_{\msF}^2}{\|\Drho\|_{\msF}^2 } + \frac{\upsilon}{2}  \cdot \frac{\|X^{\msT} \Drho\|_{\msF}^{}}{\|\Drho\|_{\msF}^{}}
 \leq 1 + \max_{0 \leq t \leq 1}\left( \frac{\upsilon^2}{4} (1 - t^2) + \frac{\upsilon}{2} t\right)
\leq (5 + \upsilon^2)/4. \nn
\end{align}
\rev{On the other hand, the relation $2J^{-1} = X^{\msT} Y + I_p$ indicates that
$\|J^{-1}\|_2 \leq 1$. Thus (iv) is true.}
}

In the case that $p=1$ or $n$, it is not difficult to simplify the corresponding update schemes and we omit the details here. Thus, we complete the proof.

\section{Proof of Lemma \ref{lemma:lowrank}}\label{appendix:section:proof2}
    \rrev{The relation $\langle G,  D^{(i)} \rangle =   e_i^{\msT}\left(G^{\msT} \nabla \mF\right) e_i$  can be easily verified from the
   definition of $D^{(i)}$}.
    With the definition of $D^{(q)}$ and $X^{\msT} X = I_p$, we have that
    $$\rrev{X^{\msT} D^{(q)} = X^{\msT} G_{(q)}^{} e_q^{\msT} - e_q G_{(q)}^{\msT}X  = \frac{2}{\tau}\left(be_q^{\msT} - e_q b^{\msT}\right)\in \mT_X}.$$
\rrev{Thus $Y(\tau)$ is well-defined. Moreover, we have that}
    $$W = -(I_n - XX^{\msT}) D^{(q)} = -(I_n - XX^{\msT}) G_{(q)}^{}e_q^{\msT}$$
    Hence, \rev{the matrix $J = I_p + \frac{\tau^2}{4} W^{\msT}W + \frac{\tau}{2}X^{\msT} D^{(q)}$ can be expressed as}
\begin{align} \label{equ:lowranJ}
J =  I_p + \xi e_q^{}e_q^{\msT} + b e_q^{\msT} - e_q b^{\msT} =
\rrev{I_p + \big[\ba{cc} e_q,&\  b \ea \big] \left(\left[\ba{cc} \xi &\  -1 \\ 1  & \ 0\ea \right] \left[\ba{c}e_q^{\msT}\\[2pt]b^{\msT}\ea\right]\right)},
\end{align}
where $\xi = \frac{\tau^2}{4}G_{(q)}^{\msT}(I_n - XX^{\msT}) G_{(q)}$.  \rrev{By the formulations of $b$ and $\xi$, we easily see that}
$$\rrev{e_q^{\msT} b = 0,  \  \xi =\alpha\, - \, b^{\msT}b.}$$
\rrev{Applying the SMW formula to \reff{equ:lowranJ} and using the above relations, we can obtain \reff{equ:lowrank:inverseJ}.}

Moreover, notice that $\nabla \mF = \sum_{i=1}^{p} D^{(i)}$. Thus we have
$$p\, \mF'_{\tau}(Y(0)) = -p\, \langle G, D^{(q)} \rangle \leq - \langle G, \nabla \mF \rangle \leq -\frac 12\|\nabla \mF\|_{\msF}^2,  $$
\rev{where the first inequality follows from the choice of $q$ in \reff{equ:lowrank_q} and the second one is due to (ii) of Lemma \ref{lemma:update}.}
\rev{The proof is completed}. 

\section{Proof of Proposition \ref{proposition:update_equ}}\label{appendix:section:proof3}
\revmini{Before going into the details of the proof, we recall the following fact on the inverse of a $2 \times 2$ block matrix.}
\begin{fact}\label{fact:blockinverse}
  Assume that $T = \left[\begin{array}{cc} T_{11} & ~T_{12}\\ T_{21} &~T_{22} \end{array}  \right]$, where
    $T_{22}^{}$ and $T_{11}^{} - T_{12}^{} T_{22}^{-1}T_{21}^{}$ are invertible. Then $T$ is invertible and
    \be
    T^{-1} =\left[ \begin{array}{cc} ( T_{11}^{} - T_{12}^{}  T_{22}^{-1}T_{21}^{})^{-1} & - ( T_{11}^{} - T_{12}^{} T_{22}^{-1}T_{21}^{})^{-1} T_{12}^{}T_{22}^{-1}
	\\[2pt]
	-T_{22}^{-1}T_{21}^{} ( T_{11}^{} - T_{12}^{} T_{22}^{-1}T_{21}^{})^{-1} &~ ~ T_{22}^{-1}T_{21}^{} ( T_{11}^{} - T_{12}^{} T_{22}^{-1}T_{21}^{})^{-1} T_{12}^{}T_{22}^{-1} + T_{22}^{-1}\end{array} \right].
      \nn
    \ee
\end{fact}

   \rev{First, we show that the update scheme \reff{equ:update:wenyin} is well-defined, provided that $I_p+\frac{\tau}{4}X^{\msT}D$ is invertible.}  Consider the update scheme \reff{equ:update:wenyin} with $U=[P_X D, \,X]$ and $V=[X,\, -P_XD]$.
   \rev{It follows from $P_X = I_n - \frac{1}{2}XX^{\msT}$ that $X^{\msT}P_X D = \frac12 X^{\msT}D$.} \rev{Combining this with $X^{\msT}X = I_p$ and $X^{\msT}D$ being skew-symmetric}, we can rewrite
\be \label{equ:proposition:update_equ:000}
I_{2p}+\frac{\tau}{2}V^{\msT}U=\left[\begin{array}{cc}I_p+\frac{\tau}{4}X^{\msT}D & \frac{\tau}{2}I_p\\[2pt] -\frac{\tau}{2}D^{\msT}P_X^{\msT}P_X^{}D &~ I_p +\frac{\tau}{4}X^{\msT}D\end{array}\right].
  \ee
\rev{Moreover, we derive  that}
$$\rev{W^{\msT}W = D^{\msT}(I_n - XX^{\msT})D = D^{\msT}P_X^{\msT} P_X^{} D + \frac{1}{4} \left(X^{\msT}D\right)^2,}$$
\rev{which with \reff{equ:update:new} implies that}
   \be
   J = \Big(I_p + \frac{\tau}{4}X^{\msT} D \Big)^2 +  \frac{\tau^2}{4}D^{\msT}P_X^{\msT}P_X^{} D. \nn
   \ee
   \rev{Plugging this into \reff{equ:proposition:update_equ:000} yields}
\be \label{equ:proposition:update_equ:a0}  I_{2p}+\frac{\tau}{2}V^{\msT}U\coloneqq   \left[\begin{array}{cc}T_{11} & ~T_{12} \\ T_{21} &~ T_{22} \end{array}\right]=\left[\begin{array}{cc}I_p+\frac{\tau}{4}X^{\msT}D & \frac{\tau}{2}I_p\\[2pt]
 \frac{2}{\tau} \left( \left(I_p +\frac{\tau}{4}X^{\msT}D\right)^2 - J\right) &~ I_p +\frac{\tau}{4}X^{\msT}D\end{array}\right],
 \ee
  where  $T_{ij} \in \mbR^{p\times p}$ ($i,j \in \{1,2\}$). \rev{By simple calculations, we know that}
  $$\rev{T_{11}^{} - T_{12}^{} T_{22}^{-1} T_{21}^{} = \left(I_p + \frac{\tau}{4}X^{\msT}D\right)^{-1} J}.$$ Thus it follows from \revmini{Fact} \ref{fact:blockinverse}  that $I_{2p}+\frac{\tau}{2}V^{\msT}U$ is invertible and  the update scheme \reff{equ:update:wenyin} is well-defined.

\rev{We now prove the equivalence.} With Lemma \ref{fact:blockinverse} and \reff{equ:proposition:update_equ:a0}, \rev{some tedious manipulations yield} that  $(I_{2p}+\frac{\tau}{2}V^{\msT}U)^{-1}=\left[\begin{array}{cc}M_{11} & ~M_{12}\\M_{21} & ~M_{22}\end{array}\right]$, where
\be
\ba{ll}
  M_{11} = J^{-1}(I_p + \frac{\tau}{4}X^{\msT} D), & ~ M_{12} = -\frac{\tau}{2}J^{-1}, \\
  M_{21} = \frac{2}{\tau}\big(I_p- \rev{M_{22}}(I_p+\frac{\tau}{4}X^{\msT}D)\big), &~ M_{22}=(I_p+\frac{\tau}{4}X^{\msT}D)J^{-1}. \nn
\ea
\ee
Direct calculations show that
\be \label{equ:proposition:update_equ:a1}
M_{11}+ \frac{1}{2}M_{12}X^{\msT}D  = J^{-1}, \quad
\tau \Big(M_{21} +  \frac{1}{2} M_{22}X^{\msT}D\Big)  = 2I_p - 2\Big(I_p+ \frac{\tau}{4}X^{\msT}D\Big)J^{-1}.
\ee
\rev{Finally, we can obtain}
\begin{align}
  Y_{\mathrm{wy}}(\tau) ={} & X - \tau \left[\ba{l}P_XD,\  X\ea\right] \Bigg(\left[\begin{array}{cc}M_{11} & ~M_{12}\\M_{21} & ~M_{22}\end{array}\right] \left[\ba{c} I_p\\ \frac{1}{2}X^{\msT}D \ea\right]\Bigg) \nn \\
    ={} & X-\tau \Big(\frac12 XX^{\msT}D -W\Big)\Big(M_{11}+ \frac{1}{2}M_{12}X^{\msT}D\Big) -\tau X\Big(M_{21}+\frac{1}{2}M_{22}X^{\msT}D\Big)\nonumber \\
    ={}& (2X + \tau W)J^{-1}- X, \nn
\end{align}
\rev{where the first equality uses \reff{equ:update:wenyin} and $X^{\msT}X = I_p$, the second and third equalities are due to  $P_XD = \frac12 XX^{\msT}D -W$ and  \reff{equ:proposition:update_equ:a1}, respectively. The proof is completed.}
\section{Some details of Table \ref{table:cost}}\label{appendix:section:details}

 We first review the computational costs of some basic matrix operations.  Given $A \in \mbR^{n_1 \times n_2}$ and $B  \in \mbR^{n_2 \times n_2}$,   computing $AB$ needs $2n_1^{}n_2^2$ flops while computing $A^{\msT} A$ only needs $n_1^{}n_2^2$ flops. Computing $\mbox{qr}(A)$ by the modified Gram-Schmidt algorithm needs $2n_1^{}n_2^2$ flops (\rev{here $n_1\geq n_2$}). If $B$ is symmetric, computing the eigenvalue decomposition  $B = P\Sigma P^{\msT}$   with orthogonal $ P\in \mbR^{n_2 \times n_2}$ and diagonal $\Sigma \in  \mbR^{n_2 \times n_2}$ by the symmetric QR  algorithm needs $9n_2^3$ flops. If $B$ is skew-symmetric, computing the exponential of  $B$ needs $10n_2^3$ flops. If $B$ is symmetric positive definite, computing $B^{1/2}$ needs $10n_2^3$ flops. In addition, if $B$ is nonsingular, solving the matrix equation $B T = A^{\msT}$  by the Gauss elimination with partial pivoting needs $2n_1^{}n_2^2 +2n_2^3/3$  flops. It follows that computing $B^{-1}$ needs $8n_2^3/3$ flops. We refer interested readers to \cite{golub1996matrix} for more details.

 \revmini{To verify the computational costs for the corresponding update scheme in Table \ref{table:cost},} we  take the new update scheme \reff{equ:update:new} and the Wen-Yin update scheme \reff{equ:update:wenyin} as examples. Notice that comparing with $O(np^2)$ or $O(n^3)$,  the $O(p^2)$ term is omitted for  $1\leq p \leq n$.  \revmini{For simplicity, we only consider  the case that $1< p <n$. In the case \brev{that} $p=1$ or $n$, the cost for the update schemes  \reff{equ:update:new} and \reff{equ:update:wenyin} can be easily obtained by a similar analysis. We omit its details here. }

 To analyze the computational cost for \reff{equ:update:new}, since $E = \Drho$ and hence $W = -G + XX^{\msT}G$, we can rewrite the feasible curve as
$$Y(\tau; X) = \left(X\Big(\frac{2}{\tau}I_p +  X^{\msT}G\Big) - G\right) {\left(\frac{J(\tau)}{\tau}\right)^{-1}} - X, $$
and
$$J(\tau) = I_p + \frac{\tau^2}{4} (G^{\msT}G - G^{\msT}XX^{\msT}G) + \rrev{\rho\tau}(X^{\msT}G - G^{\msT}X).$$
Forming $J(\tau)$ needs $3np^2 + p^3$ flops involving  computing $X^{\msT}G$, $G^{\msT}G$ and $(G^{\msT}X )(X^{\msT}G )$.    The final assembly for $Y(\tau)$ consists of involving solving one matrix equation and performing one matrix multiplication and two matrix subtractions, and hence needs  $4np^2 + 2np + 2p^3/3$ flops. Therefore, the  total  cost of  forming $Y(\tau; X)$ in  \reff{equ:update:new} for any $\tau$ is $7np^2 + 2np + 5p^3/3$.  While doing backtracking line searches, we need to update $Y(\tau; X)$ for a different $\tau$. Denote the first trial and  the new trial stepsizes by $\tau_{\brev{\mathrm{first}}}$ and  $\tau_{\brev{\mathrm{new}}}$, respectively.  It is easy to see that
$$ Y(\tau_{\brev{\mathrm{new}}}; X) = \left( X\Big(\frac{2}{\tau_{\brev{\mathrm{first}}}}I_p + X^{\msT}G\Big) - G + \Big(\frac{2}{\tau_{\brev{\mathrm{new}}}} -\frac{2}{\tau_{\brev{\mathrm{first}}}}\Big)X\right)\brev{\left(\frac{J(\tau_{\brev{\mathrm{new}}})}{\tau_{\mathrm{new}}}\right)^{-1}} - X.$$
As $X\big(\frac{2}{\tau_{\brev{\mathrm{first}}}}I_p + X^{\msT}G\big) - G$ is already computed in $Y(\tau_{\brev{\mathrm{first}}}; X)$, we store this matrix and hence only need
$2np^2 + 3np + 2p^3/3$ to compute $Y(\tau_{\brev{\mathrm{new}}}; X).$

To analyze the computational cost for the Wen-Yin update scheme \reff{equ:update:wenyin}, we see that it takes  $3np^2$  and $2np^2 + np$ flops to form $I_{2p} + \frac{\tau}{2} V^{\msT} U$ and  $P_X \Drho$, respectively. The final assembly for $Y(\tau)$ needs $4np^2 + np + 40p^3/3$ flops. Hence, the total cost for the Wen-Yin update scheme is $9np^2 + 2np + 40p^3/3 $. When  $\rho = 1/2,$ we see from \brev{\cite{wen2013feasible}} that  $U = [G, X]$, $V = [X, -G]$ in \reff{equ:update:wenyin} which implies that there is no need to compute $P_XD_{1/2}$ any more, and the total cost for forming the  Wen-Yin update scheme can be reduced to $7np^2 + np + 40p^3/3$.
The cost for updating $Y(\tau)$ with a new $\tau$ is $4np^2 + np + 40p^3/3$ since $V^{\msT}U$ can be stored and the main work \brev{involves} solving a matrix equation and the  final assembly.

\end{appendices}

\bibliography{StiefelManifoldOptAFBB}
\bibliographystyle{spmpsci}
}
\end{document}